\documentclass[3p]{elsarticle}
\usepackage[x11names,usenames,dvipsnames]{xcolor}
\usepackage{tikz}
\usepackage{comment}
\usepackage{hyperref}

\usepackage[T1]{fontenc}
\usepackage{amsthm}

\usepackage{mathtools}
\usepackage{mathdots,enumerate}
\usepackage{textcomp}
\usepackage{siunitx}
\usepackage{amssymb,amsfonts}
\usepackage{amsmath,paralist,enumitem,mathrsfs,graphicx}
\usepackage{epstopdf}
\usepackage{pgfplots}
\newcommand\norm[1]{\left\lVert#1\right\rVert}

\usetikzlibrary{pgfplots.dateplot}

\newcommand\restr[2]{{
		\left.\kern-\nulldelimiterspace 
		#1 
		\vphantom{\big|} 
		\right|_{#2} 
}}

\makeatletter
\let\save@mathaccent\mathaccent
\newcommand*\if@single[3]{%
	\setbox0\hbox{${\mathaccent"0362{#1}}^H$}%
	\setbox2\hbox{${\mathaccent"0362{\kern0pt#1}}^H$}%
	\ifdim\ht0=\ht2 #3\else #2\fi
}
\newcommand*\rel@kern[1]{\kern#1\dimexpr\macc@kerna}
\newcommand*\widebar[1]{\@ifnextchar^{{\wide@bar{#1}{0}}}{\wide@bar{#1}{1}}}
\newcommand*\wide@bar[2]{\if@single{#1}{\wide@bar@{#1}{#2}{1}}{\wide@bar@{#1}{#2}{2}}}
\newcommand*\wide@bar@[3]{%
	\begingroup
	\def\mathaccent##1##2{%
		\let\mathaccent\save@mathaccent
		\if#32 \let\macc@nucleus\first@char \fi
		\setbox\z@\hbox{$\macc@style{\macc@nucleus}_{}$}%
		\setbox\tw@\hbox{$\macc@style{\macc@nucleus}{}_{}$}%
		\dimen@\wd\tw@
		\advance\dimen@-\wd\z@
		\divide\dimen@ 3
		\@tempdima\wd\tw@
		\advance\@tempdima-\scriptspace
		\divide\@tempdima 10
		\advance\dimen@-\@tempdima
		\ifdim\dimen@>\z@ \dimen@0pt\fi
		\rel@kern{0.6}\kern-\dimen@
		\if#31
		\overline{\rel@kern{-0.6}\kern\dimen@\macc@nucleus\rel@kern{0.4}\kern\dimen@}%
		\advance\dimen@0.4\dimexpr\macc@kerna
		\let\final@kern#2%
		\ifdim\dimen@<\z@ \let\final@kern1\fi
		\if\final@kern1 \kern-\dimen@\fi
		\else
		\overline{\rel@kern{-0.6}\kern\dimen@#1}%
		\fi
	}%
	\macc@depth\@ne
	\let\math@bgroup\@empty \let\math@egroup\macc@set@skewchar
	\mathsurround\z@ \frozen@everymath{\mathgroup\macc@group\relax}%
	\macc@set@skewchar\relax
	\let\mathaccentV\macc@nested@a
	\if#31
	\macc@nested@a\relax111{#1}%
	\else
	\def\gobble@till@marker##1\endmarker{}%
	\futurelet\first@char\gobble@till@marker#1\endmarker
	\ifcat\noexpand\first@char A\else
	\def\first@char{}%
	\fi
	\macc@nested@a\relax111{\first@char}%
	\fi
	\endgroup
}

\newcommand{\dd}{\mathrm{d}}
\newcommand{\upperRomannumeral}[1]{\uppercase\expandafter{\romannumeral#1}}
\newcommand{\lowerRomannumeral}[1]{\lowercase\expandafter{\romannumeral#1}}
\usepackage{indentfirst}
\usepackage{epstopdf}
\usepackage[caption=false]{subfig}
\theoremstyle{plain}
\newtheorem{Condition}{Condition}
\newtheorem{hyp}{Hypothesis}
\newtheorem{theorem}{Theorem}
\newtheorem{lemma}[theorem]{Lemma}
\newtheorem*{lemma*}{Lemma}

\newtheorem{corollary}[theorem]{Corollary}
\newtheorem{proposition}[theorem]{Proposition}
\newtheorem{rem}{Remark}
\theoremstyle{definition}

\makeatletter
\def\ps@pprintTitle{%
	\let\@oddhead\@empty
	\let\@evenhead\@empty
	\def\@oddfoot{\footnotesize\itshape
		\ifx\@empty\@empty
		\else\@journal\fi\hfill\today}%
	\let\@evenfoot\@oddfoot	
}
\begin{document}
	 \begin{frontmatter}
	 	\title{Affine Volterra processes with jumps\tnoteref{t1}}
		\tnotetext[t1]{The research of Sergio Pulido benefited from the financial support of the chairs ``Deep finance \& Statistics'' and ``Machine Learning \& systematic methods in finance'' of \'Ecole Polytechnique. Giulia Livieri and Sergio Pulido acknowledge support by the Europlace Institute of Finance (EIF) and the Labex Louis Bachelier, research project: ``The impact of information on financial markets''. }
	 	\author[1]{Alessandro Bondi%
	 	}
	 	\ead{alessandro.bondi@sns.it}
	 	\author[2]{Giulia Livieri%
	 	}
	 	\ead{giulia.livieri@sns.it}
	 	\author[3]{Sergio Pulido%
	 	}
	 	\ead{sergio.pulidonino@ensiie.fr}
	 	\address[1]{Classe di Scienze, Scuola Normale Superiore di Pisa}
	 	\address[2]{Classe di Scienze, Scuola Normale Superiore di Pisa}
	 	\address[3]{Universit\'e Paris--Saclay, CNRS, ENSIIE, Univ \'Evry, Laboratoire de Math\'ematiques et Mod\'elisation d'\'Evry (LaMME)}
	 		\begin{abstract}
	 		The theory of affine processes has been recently extended to the framework of stochastic Volterra equations with continuous trajectories. These so--called affine Volterra processes overcome modeling shortcomings of affine processes because they can have trajectories whose regularity is different from the regularity of the paths of Brownian motion. More specifically, singular kernels yield rough affine processes. This paper extends the theory by considering affine stochastic Volterra equations with jumps. This extension is not straightforward because the jump structure together with possible singularities of the kernel may induce explosions of the trajectories. 
	 		\\[2ex] 
	 		\\[2ex]
	 	\end{abstract}
	 \end{frontmatter}

	\section{Introduction}\label{sec:intro}
	Affine stochastic processes constitute unquestionably the most popular multi--factor framework to model rich and flexible stochastic dependence structures.
	Semi--explicit formulas for the Fourier--Laplace transform of affine processes make them numerically tractable since Fourier transform--based methods can be used to perform fast calculations. 
	
	We recall that a square--integrable conservative regular affine process $X$ with state space $E\subset \mathbb{R}^m$ is a special semimartingale whose semimartingale characteristics  $(B,C,\nu)$, with respect to the ``truncation function'' $h(\xi)=\xi$, are of the form
	\begin{equation}\label{eq:BCnu}
		B_t=\int_0^t b(X_s)\,\dd s,\quad C_t=\int_0^t a(X_s)\,\dd s,\quad \nu(\dd t,\dd\xi)=\eta(X_{t},\dd\xi)\,\dd t
	\end{equation}
	where
	\begin{equation}\label{eq:defabeta}
		b(x)=b_0+\sum_{k=1}^m x_kb_k,\quad
		a(x)=A_0+\sum_{k=1}^m x_kA_k,\quad
		\eta(x,\dd\xi)=\nu_0(\dd\xi)+\sum_{k=1}^m x_k\nu_k(\dd\xi),\quad x\in E.
	\end{equation}
	In \eqref{eq:defabeta} we take $A_k\in\mathbb{R}^{m\times m}$, $b_k\in\mathbb{R}^m$, and $\nu_k(\dd\xi)$ signed measures on $\mathbb{R}^m$ such that $\nu_k(\{0\})=0$ and $\int_{\mathbb{R}^m}|\xi|^2|\nu_k|(\dd\xi)<\infty$. Additional conditions on the parameters $A_k$, $b_k$ and $\nu_k$ have to imposed in order to guarantee existence and invariance results depending on the state space $E$. See for instance \cite{DFS} for $E=\mathbb{R}_+^k\times\mathbb{R}^l$ and \cite{CFMT} for $E$ equal to the space of positive semidefinite matrices.
	
	The Fourier--Laplace transform of such an affine process $X$ is given by 
	\begin{equation}\label{eq:FourierLaplaceclassical}
		\mathbb{E}\left[\exp\left(\int_0^T f(T-s)^\top  X_s\,\dd s\right)\bigg |\mathcal{F}_t\right]=\exp\left(\phi(T-t)+\int_0^t f(T-s)^\top  X_s\,\dd s+ \psi(T-t)^\top X_t\right)
	\end{equation}
	with  $\psi$ a $\mathbb{C}^m-$valued function that solves the Riccati equation
	\begin{equation}\label{eq:Riccati}
		\psi(t)=\int_0^t F\left(s,\psi(s)\right)\dd s
	\end{equation}
	where
	\begin{equation}\label{eq:F}
		F_k\left(s,z\right)=f_k\left(s\right)+\frac12z^\top A_k z+ z^\top b_k  + \int_{\mathbb{R}^m}\left({\rm e}^{ z^\top \xi}-1-z^\top \xi \right)\nu_k(\dd\xi),\quad k=1,\ldots,m,
	\end{equation}
	and $\phi$ the $\mathbb{C}-$valued function given by 
	\begin{equation}\label{eq:phiRiccati}
		\phi(t)=\int_0^t \left(\psi(s)^\top b_0 + \frac12\psi(s)^\top A_0 \psi(s) + \int_{\mathbb{R}^m}\left({\rm e}^{ \psi(s)^\top \xi}-1-\psi(s)^\top \xi \right)\nu_0(\dd\xi)\right)\dd s.
	\end{equation}
	The identity \eqref{eq:FourierLaplaceclassical} is only valid under additional hypotheses on the $\mathbb{C}^m-$valued function $f$ and $t,T\ge 0$ that imply appropriate conditions on the functions $\phi$ and $\psi$. \footnote{One of those conditions could be for instance the boundness of the right term in \eqref{eq:FourierLaplaceclassical}. On a related subject, we also refer to \cite{KR} where the authors analyze the possible explosions of the associated Riccati equation \eqref{eq:Riccati}.}

	The theory of affine processes was recently extended in \cite{sergio,GK} to the framework of stochastic Volterra equations with continuous trajectories, where in general the semimartingale and Markov properties do not hold. These so--called affine Volterra processes overcome modeling shortcomings of affine processes because they can have trajectories whose regularity is different from the regularity of the paths of Brownian motion. More specifically, singular kernels yield rough processes in the spirit of \cite{BFG,gath,el}. The goal of this paper is to extend the results in \cite{sergio} by considering affine stochastic Volterra equations with jumps. This extension is not straightforward because the jump structure together with possible singularities of the kernel may induce explosions of the trajectories. 
	
	Our study can be motivated by financial models for stock volatility. In particular the observation in \cite{to} that a complete description of volatility should take into account both path roughness and jumps; see also \cite{xi} for an interesting discussion on the topic. In this paper, however, we concentrate on the mathematical properties of this family of processes and we address their possible applications in a separate article \cite{BLP}.
	
	We summarize in this introduction the framework and the main results of our study. Suppose that $X$ is a predictable solution to a stochastic Volterra equation of the form
	\begin{equation}\label{eq:SVE}
		X_t = g_0(t) + \int_{0}^t K(t-s)\,\dd Z_s,\quad \mathbb{P}\otimes\dd t\text{--a.e.}
	\end{equation}
	defined on a filtered probability space $(\Omega,\mathcal F,(\mathcal F_t)_{t\ge 0},\mathbb{P})$ satisfying the usual conditions and with trajectories in  $L^1_\text{loc}(\mathbb{R}_+;E)$ for some state space $E\subset \mathbb{R}^m$. In \eqref{eq:SVE} we take $g_0\in L^1_\text{loc}(\mathbb{R}_+;\mathbb{R}^m)$, $K\in L^2_\text{loc}(\mathbb{R}_+;\mathbb{R}^{m\times d})$ a matrix--valued kernel, and $Z$ a $d-$dimensional semimartingale whose characteristics depend on $X$. In order to have an affine structure we suppose that $Z$ has characteristics of the form \eqref{eq:BCnu}--\eqref{eq:defabeta}, with $A_k\in\mathbb{R}^{d\times d}$, $b_k\in\mathbb{R}^d$, and $\nu_k$ signed measures on $\mathbb{R}^d$ such that $\nu_k(\{0\})=0$ and $\int_{\mathbb{R}^d}|\xi|^2|\nu_k|(\dd\xi)<\infty$. In this case we call $X$ an affine Volterra process. When $E=\mathbb{R}^m$, existence of weak solutions to \eqref{eq:SVE} with trajectories in $L^2_\text{loc}(\mathbb{R}_+;\mathbb{R}^m)$ has been established in \cite[Theorem 1.2]{ACLP}. For $E=\mathbb{R}_+$ and for a Volterra CIR--type of process with positive jumps results in this direction can be found in \cite[Theorem  2.13]{edu}. 
	
	Section \ref{sec:prelim} contains more details about this setting as well as important results regarding stochastic convolutions with respect to processes with jumps. These results, which play a crucial role in our arguments, extend and are inspired by those in \cite{sergio} where the authors study stochastic convolutions with respect to processes with continuous trajectories.
	
	Fix $T>0$, $f\in C\left(\mathbb{R}_+;\mathbb{C}^m\right)$. By analogy with \eqref{eq:Riccati} assume that $\psi\in C(\mathbb{R}_+;\mathbb{C}^d)$ solves the Riccati--Volterra equation
	\begin{equation}\label{eq:RiccatiVolterra}
		\psi\left(t\right)^\top=\int_0^t F\left(s,\psi(s)\right)^\top K(t-s)\,\dd s,
	\end{equation} 
	with $F$ as in \eqref{eq:F}, and let $\phi$ be given by \eqref{eq:phiRiccati}.
	
	In Section \ref{sec:FLT} we show our first main result, namely Theorem \ref{t6}, which is a generalization of \cite[Theorem 4.3]{sergio} and provides a semi--explicit formula for the Fourier--Laplace transform of $X$. This theorem shows that under the above--mentioned framework if we define 
	\begin{equation}\label{eq:M}
		M_t=\exp\left(\phi(T-t)+\int_0^t f(T-s)^\top X_s\,\dd s+\int_t^T F\left(T-s,\psi(T-s)\right)^\top g_t\left(s\right)\dd s\right)
	\end{equation}
	where $\left(g_t\left(\cdot\right)\right)_{t\ge 0}$ denotes the adjusted forward process\footnote{This adjusted forward process was also used in \cite{AE} and \cite{KLP} to elucidate the affine structure of affine Volterra processes with continuous trajectories.}
	\begin{equation}\label{eq:fwdprocess}
		g_t(s)=g_0(s)+\int_0^t K\left(s-r\right)\,\dd Z_r,\quad s>t,
	\end{equation}
	then $M$ is a local martingale, and if  $M$ is a martingale then one has the exponential--affine formula
	\begin{equation}\label{eq:FourierLaplace}
		\mathbb{E}\left[\exp\left(\int_0^T f(T-s)^\top  X_s\,\dd s\right)\bigg |\mathcal{F}_t\right]=M_t.
	\end{equation}
	As a consequence, under these conditions, uniqueness in law holds for the stochastic Volterra equation \eqref{eq:SVE}.
	
	Section \ref{sec:past} contains our second main result which is Theorem \ref{main}. This theorem establishes, under the assumption $m=d$ and additional conditions on the kernel $K$, an alternative formula for the local martingale $M_t$ in \eqref{eq:M} in terms of $(X_s)_{s\le t}$ and $Z_t$ only, namely
	\begin{multline}\label{eq:past}
		\log\left(M_t\right) = \phi\left(T-t\right)+\int_0^t f\left(T-s\right)^\top X_s\,\dd s +\int_0^{T-t}F\left(s,\psi\left(s\right)\right)^\top g_0\left(T-s\right)\dd s \\
		+\psi\left(T-t\right)^\top Z_t +\left(\pi_{T-t}^\top \ast \left(X-g_0\right)\right)\left(t\right),
	\end{multline} 
	with $\phi$ as in \eqref{eq:phiRiccati} and $\pi_h\in L_\text{loc}^1(\mathbb{R}_+;\mathbb{C}^d),\,h>0,$ a deterministic function that depends on $K$ and $\psi$. This expression is a corollary of a similar expression for the adjusted forward process \eqref{eq:fwdprocess}, shown in Lemma \ref{l7}.  The identity \eqref{eq:past} can be used to show that \eqref{eq:FourierLaplaceclassical} is a particular instance of \eqref{eq:FourierLaplace} when  $g_0$ is constant, and $K$ is constant and equal to the identity matrix.

	In Section \ref{sec:example}, using our first two main results, we give a complete proof in Theorems \ref{tadd} and \ref{bound_t} of the exponential--affine formula \eqref{eq:FourierLaplace} in the particular case $m=d=1$, $E=\mathbb{R}_+$ and for a Volterra CIR--type processes with positive jumps. The argument hinges on a novel comparison result between solutions of Riccati--Volterra equations, namely between a solution of \eqref{eq:RiccatiVolterra} and a solution of an analogous equation in which the functions $\psi$ and $F$ are substituted with the corresponding real parts. This comparison result, together with the \textit{affine with respect to the past formula} \eqref{eq:past} of Theorem \ref{main}, yields the desired conclusion because we can bound the complex--valued local martingale $M$ \eqref{eq:M} of Theorem \ref{t6} with a real--valued martingale. \footnote{It is important to mention at this point the study in \cite{CT} where the authors construct infinite dimensional lifts of affine Volterra processes, possibly with jumps, and study affine transform formulas for these lifts. Even though these formulas are closely related to the affine transform formulas of the present study,  the novelty of our work stems from the \textit{affine with respect to the past formula} \eqref{eq:past} and the complete analysis of the associated Riccati equations in the one--dimensional case as it is carried out in Section \ref{sec:example}. Our approach is inspired by the arguments in \cite{sergio}, and extends them to a jump--processes framework.}
	
	\ref{ap_A} contains some basic results regarding the classical forward process and \ref{ap_B} results regarding the $1-$dimensional Riccati--Volterra equations appearing in Section \ref{sec:example}.
	\\[2ex]
	\noindent\textbf{Notation:} Throughout the paper, elements of $\mathbb{R}^k$ and $\mathbb{C}^k$ are column vectors. 
	Given a matrix $A\in\mathbb{C}^{k\times l},$ the element in row $i$ and column $j$ is $A^{ij}$, $A^\top\in\mathbb{C}^{l\times k}$ is its transpose matrix, and $\left|A\right|$ is the Frobenius norm.  
	We also use the notation $\mathbb{R}^k_+=\left\{x\in\mathbb{R}^k : x_i\ge0,\,i=1,\dots,k\right\}$ and $\mathbb{C}^k_-=\left\{x\in\mathbb{C}^k : \mathfrak{R}\left(x_i\right)\le0,\,i=1,\dots,k\right\}$, where for $z\in\mathbb{C}$, $\mathfrak{R}(z)$ denotes its real part. The imaginary part of a complex number $z$ is $\mathfrak{Im}z$. We use the convolution notation $\left(f\ast g\right)\left(t\right)=\int_0^t f\left(t-s\right)\,g\left(s\right)\dd s$ for functions $f,g$.

	\section{Preliminaries}\label{sec:prelim}
	Fix $d,m\in\mathbb{N}$. Let $g_0\in L^1_{\text{loc}}\left(\mathbb{R}_+;\mathbb{R}^{m}\right)$, $K\in L^2_{\text{loc}}\left(\mathbb{R}_+;\mathbb{R}^{m\times d}\right)$ be a matrix--valued kernel and $E\subset \mathbb{R}^m$ be a subset which will be the state--space that we consider. We also introduce  a \emph{characteristic triplet} $\left(b,a,\eta\right)$ consisting of the measurable maps $b\colon \mathbb{R}^m\to \mathbb{R}^d$, $a\colon\mathbb{R}^m\to\mathbb{R}^{d\times d}$ and the transition kernel $\eta\left(x,\dd \xi\right)$ from $\mathbb{R}^m$ to $\mathbb{R}^d$.
	We require this triplet to be affine on $E$, meaning that, for every $x\in E$, 
	\begin{equation}\label{aff}
		b\left(x\right)= b_0+\sum_{k=1}^{m}x_kb_k,\qquad  a\left(x\right)= A_0+\sum_{k=1}^{m}x_kA_k,\qquad \eta\left(x,\dd \xi\right)=\nu_0\left(\dd \xi\right)+\sum_{k=1}^{m}x_k\nu_k\left(\dd \xi\right).
	\end{equation}
	Here $b_0,\,b_1,\dots,\,b_{m}\in\mathbb{R}^d$, $A_0,\,A_1,\dots,\,A_m\in\mathbb{R}^{d\times d}$, and   $\left(\nu_k\right)_{k=0,\dots,m}$ are signed measures on $\mathbb{R}^d$ such that $\int_{\mathbb{R}^d}\left|\xi\right|^2\left|\nu_{k}\right|\left(\dd \xi\right)<\infty$, with $\nu_{k}\left(\left\{0\right\}\right)=0$.
	Throughout the paper, we denote by $X=\left(X_t\right)_{t\ge0}$ a
	predictable process  with trajectories  in $L^1_{\text{loc}}\left(\mathbb{R}_+;\mathbb{R}^{m}\right)$  and such that $X\in E,\,\mathbb{P}\otimes \dd t-$a.e. It is defined on a filtered probability space $\left(\Omega, \mathcal{F},\mathbb{F},\mathbb{P}\right)$ where the filtration $\mathbb{F}=\left(\mathcal{F}_t\right)_{t\ge0}$ satisfies the \emph{usual conditions} and $\mathcal{F}_0$ is the trivial $\sigma-$algebra on $\Omega$. Moreover, 
	we assume that $X$ solves the following affine stochastic Volterra equation of convolution type
	\begin{equation}\label{f}
		X_t=g_0\left(t\right)+\int_{0}^{t}K\left(t-s\right)\dd Z_s,\quad \mathbb{P}-\text{a.s., for a.e. }t\in\mathbb{R}_+.
	\end{equation}
	Here $Z$ is a $d-$dimensional semimartingale starting at $0$ whose differential characteristics  with respect to the Lebesgue measure are $\left(b\left(X_t\right),a\left(X_t\right),{\eta}\left(X_t,\dd \xi\right)\right),\,t\ge0.$ These characteristics are taken with respect to the ``truncation function'' $h\left(\xi\right)=\xi,\,\xi\in\mathbb{R}^d$, which can be chosen because $Z$ is a special semimartingale due to \cite[Proposition $2.29$, Chapter \upperRomannumeral{2}]{js} and the local integrability of the trajectories of $X$. In the sequel, we denote by $\mu\left(\dd t,\dd \xi\right)$ the measure associated with the jumps of $Z$ and by $\nu\left(\dd t,\dd \xi\right)=\eta\left(X_t,\dd \xi\right)\dd t$ its compensator.
	
	It is worth discussing the good definition of the stochastic integral in \eqref{f}.	Recalling that $X\in E,\,\mathbb{P}\otimes \dd t-$a.e., the {canonical representation theorem for semimartingales} (see \cite[Proposition $2.34$, Chapter \upperRomannumeral{2}]{js}) shows that  $Z$ admits the decomposition 
	\[
	Z_t=\int_{0}^{t}b\left(X_s\right)\dd s+M^c_t+M^d_t=
	b_0t+\sum_{k=1}^{d}b_k\int_{0}^{t}X_{k,s}\,\dd s+M^c_t+M^d_t,\quad t\ge0,\,\mathbb{P-}\text{a.s.},
	\] 
	where  $\dd M^d_t=\int_{\mathbb{R}^d}\xi\left({\mu}-{\nu}\right)\left(\dd t,\dd \xi\right)$ is an $\mathbb{R}^d-$valued, purely discontinuous local martingale and $M^c$ is a $d-$dimensional, continuous local martingale satisfying $\dd\left\langle M^c,M^c\right\rangle_t=a\left(X_t\right)\dd t$. Now if we introduce, for every $j=1,\dots,d$, the increasing process $C^j_t=\int_{0}^{t}\int_{\mathbb{R}^d}\left|\xi_j\right|^2{\nu}\left(\dd s,\dd \xi\right),\,t\ge0$,
	then we have
	\[ 
	C^j_t=
	\left(\int_{\mathbb{R}^d}\left|\xi_j\right|^2\nu_0\left(\dd \xi\right)\right)t
	+\sum_{k=1}^m\int_{\mathbb{R}^d}\left|\xi_j\right|^2\nu_k\left(\dd\xi\right)
	\left(\int_{0}^{t}X_{k,s} \,\dd s\right),\quad t\ge0,\,\mathbb{P}-\text{a.s.}
	\]
	As a consequence of this expression, the local integrability of the paths of $X$ implies that $C^j$ is locally integrable. 
	Hence \cite[Theorem $1.33$ (a), Chapter \upperRomannumeral{2}]{js} yields that $M^d$ is a locally square--integrable martingale  with
	\begin{equation}\label{pred_q}
		\dd\left\langle M^{d}_j,M^{d}_j\right\rangle_t=
		\left[\int_{\mathbb{R}^d}\left|\xi_j\right|^2\nu_0\left(\dd\xi\right)+
		\sum_{k=1}^m\left(\int_{\mathbb{R}^d}\left|\xi_j\right|^2\nu_k\left(\dd \xi\right)\right)X_{k,t}
		\right]\dd t,
	\end{equation} 
	where $M^{d}_j$ is the $j-$th component of $M^d,\,j=1,\dots,d.$
	It is convenient to introduce the locally square--integrable martingale $\widetilde{Z}= M^c+M^d$, which satisfies \begin{equation}\label{Z}
		\widetilde{Z}_t=Z_t-\int_{0}^{t}b\left(X_s\right)\,\dd s=Z_t-b_0t-\sum_{k=1}^{d}b_k\int_{0}^{t}X_{k,s}\,\dd s,\quad t\ge0, \,\mathbb{P}-\text{a.s.}
	\end{equation}
	
	Given an integer $l\in\mathbb{N}$
	and $F\in L^2_{\text{loc}}\left(\mathbb{R}_+;\mathbb{R}^{l \times d}\right)$, we define the $l-$dimensional random variable 
	\[
	\left(F\ast \dd \widetilde{Z}\right)_T=\left(F\ast \dd M^c\right)_T+\left(F\ast \dd M^d\right)_T= \int_{0}^{T}F\left(T-s\right)\dd M^c_s+\int_{0}^{T}F\left(T-s\right)\dd M^d_s.
	\]
	This is well--defined for a.e. $T\in\mathbb{R}_+$. Indeed, consider the stopping times $\tau_n=\inf\left\{t\ge0 : \int_{0}^{t}\left|X_s\right|\dd s>n\right\}$ for all $n\in\mathbb{N}$. Since $X_\cdot\left(\omega\right)\in L^1_\text{loc}\left(\mathbb{R}_+,\mathbb{R}^m\right)$, $\tau_n\to \infty$ as $n\to\infty$ in $\Omega$. Then for every $\widebar{T}>0$, we can apply the Young's type inequality in \cite[Lemma A$.1$]{ACLP} with $p=q=r=1$ 	and Tonelli's theorem to deduce that
	\begin{multline*}
		\int_{0}^{\widebar{T}}\left(\mathbb{E}\left[\int_{0}^{T\wedge \tau_n}\left|F\left(T-s\right)\right|^2\left|X_{k,s}\right|\dd s\right]\right)\dd T=
		\int_{0}^{\widebar{T}}\left(\int_{0}^{T}\left|F\left(T-s\right)\right|^2\mathbb{E}\left[1_{\left\{s\le \tau_n\right\}}\left|X_{k,s}\right|\right]\dd s\right)\dd T\\
		\le 
		\norm{F}^2_{L^2\left(\left[0,\widebar{T}\right];\mathbb{R}^{l\times d}\right)}\mathbb{E}\left[\int_0^{\widebar{T}\wedge \tau_n}\left|X_{k,s}\right|\dd s\right]\le n\norm{F}^2_{L^2\left(\left[0,\widebar{T}\right];\mathbb{R}^{l\times d}\right)} <\infty,\quad k=1,\dots,m.
	\end{multline*}
	This ensures that $\mathbb{E}\left[\int_{0}^{T\wedge \tau_n}\left|F\left(T-s\right)\right|^2X_{k,s}\,\dd s\right]<\infty,\,k=1,\dots,m,\,n\in\mathbb{N}$, for a.e. $T\in \mathbb{R}_+$, say for every $T\in\mathbb{R}_+\setminus N$, where $N\subset \mathbb{R}_+$ is a $\dd t-$null set. As a consequence, it is straightforward to conclude that the processes 
	\begin{equation}\label{mart}
		\left(\int_0^tF\left(T-s\right)\dd M^c_s\right)_{t\in\left[0,T\right]},\qquad \left(\int_0^tF\left(T-s\right)\dd M^d_s\right)_{t\in\left[0,T\right]},
	\end{equation}
	are locally square--integrable martingales for every $T\in\mathbb{R}_+\setminus N$. Indeed, for every $n\in\mathbb{N},$
	\begin{equation*}
		\sum_{j=1}^d\mathbb{E}\left[\int_{0}^{T\wedge\tau_n}\left|F\left(T-s\right)\right|^2\dd\left\langle M^{c}_j,M^{c}_j\right\rangle_s\right]=
		\sum_{j=1}^d\mathbb{E}\left[\int_{0}^{T\wedge\tau_n}\left|F\left(T-s\right)\right|^2\left(A_0^{jj}+\sum_{k=1}^{m}X_{k,s}A_k^{jj}\right)\dd s\right]
		<\infty,
	\end{equation*}
	and (by \eqref{pred_q})
	\begin{multline*}
		\sum_{j=1}^d\mathbb{E}\left[\int_{0}^{T \wedge\tau_n}\left|F\left(T-s\right)\right|^2\dd\left\langle M^{d}_j,M^{d}_j\right\rangle_s\right]\\
		=
		\sum_{j=1}^d\mathbb{E}\left[\int_{0}^{T\wedge\tau_n}\left|F\left(T-s\right)\right|^2\left(\int_{\mathbb{R}^d}\left|\xi_j\right|^2\nu_0\left(\dd\xi\right)+
		\sum_{k=1}^mX_{k,s}\int_{\mathbb{R}^d}\left|\xi_j\right|^2\nu_k\left(\dd\xi\right)\right)\dd s\right]<\infty.
	\end{multline*}
	We always work with a jointly measurable version of the stochastic convolution $F\ast \dd \widetilde{Z}$ defined on $\Omega\times\mathbb{R}_+$ (such a modification exists, see, e.g., \cite[Theorem $3.5$]{r}).\\
	As for the convolution of $F$ with the drift part of $Z$, using \cite[Theorem $2.2$\,(\lowerRomannumeral{1}), Chapter $2$]{g} we compute
	\begin{align*}
		\begin{split}
			&\mathbb{E}\left[\int_{0}^{T}\left(\int_{0}^{T}1_{\left\{t\le \tau_n\right\}}1_{\left\{s\le t\right\}}\left|F\left(t-s\right)\right|\left(\left|b_0\right|+\sum_{k=1}^m\left|b_k\right|\left|X_{k,s}\right|\right)\dd s\right)\dd t\right]
			\\
			&=\mathbb{E}\left[\int_{0}^{T\wedge \tau_n}\left(\int_{0}^{t}\left|F\left(t-s\right)\right|\left(\left|b_0\right|+\sum_{k=1}^m\left|b_k\right|\left|X_{k,s}\right|\right)\dd s\right)\dd t\right]
			\\
			&\le\norm{F}_{L^1\left(\left[0,T\right];\mathbb{R}^{l\times d}\right)}\left[\left|b_0\right|T+n\left(\sum_{k=1}^m\left|b_k\right|\right)\right]<\infty,\quad T>0.
		\end{split}	
	\end{align*}
	This shows that there exists a $\mathbb{P}\otimes \dd t-$null set $N_1\subset \Omega\times \mathbb{R}_+$ such that the next expression is well--defined
	\[
	1_{\left\{t\le \tau_n\left(\omega\right)\right\}}\int_{0}^{t}F\left(t-s\right)\left(b_0+\sum_{k=1}^mb_kX_{k,s}\left(\omega\right)\right)\dd s,\quad n\in\mathbb{N},\,\left(\omega,t\right)\in \left(\Omega\times \mathbb{R}_+\right)\setminus N_1.
	\]
	Moreover, by Fubini's theorem the resulting processes are jointly measurable in $\left(\Omega\times \mathbb{R}_+\right)\setminus N_1$, hence passing to the limit as $n\to\infty$, we obtain the jointly measurable process $\int_{0}^tF\left(t-s\right)\left(b_0+\sum_{k=1}^mb_kX_{k,s}\right)\dd s$ (defined on the same set). Finally we introduce
	\[
	\left(F\ast b\left(X\right)\right)\left(\omega,t\right)=\begin{cases}
		\int_{0}^tF\left(t-s\right)\left(b_0+\sum_{k=1}^mb_kX_{k,s}\left(\omega\right)\right)\dd s,&\left(\omega,t\right)\in\left(\Omega\times \mathbb{R}_+\right)\setminus N_1,\\
		0,&\left(\omega,t\right)\in N_1.\\
	\end{cases}
	\] 
	This is a jointly measurable process defined on the whole $\Omega\times \mathbb{R}_+.$ This machinery for constructing jointly measurable modifications of given processes will be used several times in the sequel.\\
	Overall, the previous argument proves that the integral on the right side of \eqref{f} is well--defined $\mathbb{P}-$a.s., for a.e. $t\in\mathbb{R}_+$.  Of course we denote by $\left(F\ast \dd Z\right)= \left(F\ast b\left(X\right)\right)+\left(F\ast \dd \widetilde{Z}\right)$; with this notation, Equation \eqref{f} can be written as follows
	\begin{equation}\label{2}
		X=g_0+\left(K\ast \dd Z\right)=g_0+\left(K\ast b\left(X\right)\right)+\left(K\ast \dd \widetilde{Z}\right),\quad \mathbb{P}\otimes \dd t-\text{a.e.}
	\end{equation}
	
The following lemma will be useful in the sequel.
	\begin{lemma}
		For every $T>0$, 
		\begin{equation}\label{1.4}
			\mathbb{E}\left[\norm{X}_{L^1\left(\left[0,T\right];\mathbb{R}^m\right)}\right]<\infty.		
		\end{equation}
	\end{lemma}
	\begin{proof}
		The proof follows the same steps as those in \cite[Theorem $1.4$]{ACLP}. The difference is that the affine structure of our model guaranteed by \eqref{aff} is substituted for \cite[Condition $\left(1.5\right)$]{ACLP}, and makes  the $L^1_\text{loc}-$integrability of the paths of $X$ sufficient (instead of the $L^p_\text{loc}-$integrability, $p\ge2$ required in \cite{ACLP}). 
	\end{proof}
	Knowing the additional property in \eqref{1.4}, the same argument as the one above (without stopping times) shows that the processes in \eqref{mart} are indeed square--integrable martingales for a.e. $T\in\mathbb{R}_+$.
	\begin{rem}
		We refer to \cite{ACLP} for a general solution theory concerning equations of the type in \eqref{2} when $g_0\in L^p_{\text{loc}}\left(\mathbb{R}_+;\mathbb{R}^m\right),\,p\ge 2,$ and $E=\mathbb{R}^m$.\\
		In the case $m=d=1$ and $E=\mathbb{R}_+$, if one  defines $Y_t=\int_{0}^{t}X_s\,\dd s,\,t\ge0$, then $Y=\left(Y_t\right)_{t\ge 0}$ is a nondecreasing process and an application of \cite[Lemma $3.2$]{ACLP} shows 
		\[
		Y_t=\int_{0}^{t}g_0\left(s\right)\dd s+\int_{0}^{t}K\left(t-s\right)Z_s\,\dd s=\int_{0}^{t}g_0\left(s\right)\dd s+\left(K\ast Z\right)_t,\quad t\ge0,\, \mathbb{P}-\text{a.s.}
		\]
		This type of stochastic Volterra equations is analyzed in \cite{edu} for locally integrable kernels $K\in L^1_\emph{loc}\left(\mathbb{R}_+;\mathbb{R}\right).$
	\end{rem}
	\subsection{Stochastic convolution for processes with jumps}
	The goal of this subsection is to develop technical results concerning the stochastic convolution. In particular, we aim to make Lemma $2.1$ and Lemma $2.6$ in \cite{sergio} feasible in our context, where we are dealing with discontinuities for $Z$ and, more importantly, with a process $X$ which a priori is not bounded. This requires the statements and the proofs of the aforementioned results --crucial for the development of the theory-- to be changed. Such changes are important from a conceptual point of view and after every result we add a remark showing the parallel with the setting in \cite{sergio}.
	
	We start with a preliminary claim.
	\begin{lemma}\label{l1}
		Fix $p\in\mathbb{N}$. Let $F,G\in L^2_{\emph{loc}}\left(\mathbb{R_+;\mathbb{R}}^{p\times d}\right)$ and $S\subset \mathbb{R}_+$ be such $\mathbb{R}_+\setminus S$ is $\dd t-$null set. Suppose that   $F=G$ a.e. in $\mathbb{R}_+$. Then
		\begin{equation}\label{.1}
			\int_{0}^{T}F\left(T-s\right) \dd Z_s=	\int_{0}^{T}1_S\left(s\right)G\left(T-s\right) \dd Z_s,\quad\mathbb{P}-\text{a.s., for a.e. }T\in\mathbb{R}_+.
		\end{equation}
	In particular,
			\begin{equation}\label{2.1}
		F\ast \dd Z=G\ast \dd Z,\quad \mathbb{P}\otimes \dd t-\text{a.e.}
	\end{equation}
	\end{lemma}
	\begin{proof}
		It is sufficient to prove \eqref{.1} replacing $Z$ with $\widetilde{Z}$, because trivially $F\ast b\left( X\right)=G\ast \left(1_Sb\left(X\right)\right),$ $\mathbb{P}\otimes \dd t-$a.e. on $\Omega\times \mathbb{R}_+$. Moreover, we only work with the stochastic integral in $\dd M^c$, as by \eqref{pred_q} we can repeat the next procedure (component--wise) for the convolution in $\dd M^d$ recovering \eqref{.1}. 
		
		The argument above in the section implies the existence of a $\dd t-$null set $N\subset \mathbb{R}_+$ such that, for every $T\in\mathbb{R}_+\setminus N$, we have 
		\begin{equation}\label{2.2}
			\int_{0}^{T}1_S\left(s\right)G\left(T-s\right)\dd M^c_s-\int_{0}^{T}F\left(T-s\right) \dd M^c_s=
			\int_{0}^{T}\left(1_S\left(s\right)G\left(T-s\right)-F\left(T-s\right)\right)\dd M^c_s,\quad
			\mathbb{P}-\text{a.s.}
		\end{equation}		
		Consider the  square--integrable, $p-$dimensional martingale $Q=\left(\int_{0}^{t}\left(1_S\left(s\right)G\left(T-s\right)-F\left(T-s\right)\right)\dd M_s^c\right)_{t\le T}$, whose predictable quadratic covariation is, due to the hypotheses,
		\[
		\left\langle Q,Q \right\rangle_t\!=\!\int_{0}^t\!\left(1_S\left(s\right)G\left(T-s\right)-F\left(T-s\right)\right)a\left(X_s\right)\left(1_S\left(s\right)G\left(T-s\right)-F\left(T-s\right)\right)^\top \dd s=0,\!\quad t\in\left[0,T\right],\mathbb{P}-\text{a.s.}
		\]
		Since $Q$ starts at $0$, we can conclude that $Q=0$ up to evanescence, hence \eqref{.1} follows.
		
	Regarding \eqref{2.1}, it is an immediate consequence of \eqref{.1} with $S=\mathbb{R}_+$ and the joint measurability of the stochastic convolutions, which allows to state an equality holding true $\mathbb{P}\otimes \dd t-$a.e. This completes the proof.
\end{proof}
	\begin{rem}\label{rem1}
		In \cite{sergio}, the authors consider the stochastic convolution of a function $F\in L^2_{\emph{loc}}\left(\mathbb{R}_+;\mathbb{R}^{p\times d}\right)$ with respect to a continuous local martingale $M$ with predictable quadratic covariation  $d\left\langle M,M \right\rangle_t=a_t\,\dd t$, where $\left(a_t\right)$ is an adapted, locally bounded  process. These assumptions allow to define $\left(F\ast \dd M\right)_t$ for every $t\in\mathbb{R}_+$. In particular, two  jointly measurable  versions of the stochastic convolution are equal $\mathbb{P}-$a.s., for every $t\ge0$. This concept is stronger than the $\mathbb{P}\otimes \dd t-$uniqueness that we have in our framework. As for \eqref{2.1} in Lemma~\ref{l1}, in the continuous case it can be stated as follows: for every $F,G\in L^2_{\emph{loc}}\left(\mathbb{R}_+;\mathbb{R}^{p\times d}\right)$, with $F=G$ a.e. in $\mathbb{R}_+$, one has
		\[
		\left(F\ast \dd M\right)_t=\left(G\ast \dd M\right)_t,\quad \mathbb{P}-\text{a.s., } t\ge0.
		\]
	\end{rem}
	Now we state a result concerning the associativity of the stochastic convolution.
	\begin{lemma}\label{4.1}
		Fix $p,q\in\mathbb{N}$. Let $\rho\in L^1_\emph{loc}\left(\mathbb{R}_+;\mathbb{R}^{q\times p}\right)$ and $F\in L^2_{\emph{loc}}\left(\mathbb{R}_+;\mathbb{R}^{p\times d}\right)$. Then
		\begin{equation}\label{time}
			\left(\left(\rho\ast F\right)\ast \dd Z\right)_t=\left(\rho\ast \left(F\ast \dd Z\right)\right)\left(t\right),\quad \mathbb{P}-\text{a.s., for a.e. }t\in\mathbb{R}_+.
		\end{equation}
	\end{lemma}
	\begin{proof}
		Also in this case we just need to show the statement with $\dd \widetilde{Z}$ in place of $\dd Z$, because an application of Fubini's theorem provides
		$\left(\rho\ast F\right)\ast b\left(X\right)=\rho\ast \left(F\ast b\left(X\right)\right),\,\mathbb{P}\otimes \dd t-$a.e. on $\Omega\times\mathbb{R}_+$.
		In addition it is sufficient to focus only on the stochastic convolutions in $\dd M^c$, as discussed in the preceding proof. By linearity we can assume $d=p=q=1$ without loss of generality, and we consider $\rho\ge0$ to keep the notation simple, otherwise we should split it into positive and negative part. 
		
		 First note that the function $\rho\ast F\in L^2_\text{loc}\left(\mathbb{R}_+;\mathbb{R}\right)$, hence for every  $t\in\mathbb{R}_+\setminus N_1$, being $N_1$ a $\dd t-$null set, we have
		\begin{multline}\label{he}
			\left(\left(\rho\ast F\right)\ast \dd M^c\right)_t=\int_{0}^t \left(1_{\left\{\left(t-u\right)\in S\right\}}\int_{0}^{t-u}F\left(t-u-s\right)\rho\left(s\right)\dd s\right)\dd M^c_u\\
			=\int_{0}^t\left(\int_{0}^{t}1_{\left\{\left(t-u\right)\in S\right\}}1_{\left\{s\le t-u\right\}}F\left(t-u-s\right)\rho\left(s\right)\dd s\right)\dd M^c_u,\quad \mathbb{P}-\text{a.s},
		\end{multline}
		where $S\subset\mathbb{R}_+$ is such that $\int_{0}^{t}F\left(t-s\right)\rho\left(s\right)\dd s ,\,t\in S,$ is well--defined. In particular, $\mathbb{R}_+\setminus S$ is a $\dd t-$null set.
		Our goal is to apply the stochastic Fubini's theorem (see, e.g., \cite[Theorem $65$, Chapter \upperRomannumeral{4}]{protter}), but before we can do that we need a preliminary step. For every $T>0$,  a change of variables, sequential applications of Tonelli's theorem and Young's inequality yield (in the whole $\Omega$)
		\begin{align*}
			\begin{split}
				&\int_{0}^{T}\left[\int_0^t\left(\int_0^t1_{\left\{\left(t-u\right)\in S\right\}}1_{\left\{s\le t-u\right\}}\left|F\left(t-s-u\right)\right|^2\rho\left( s\right)\dd s\right)\left|X_u\right|\,\dd u\right]\dd t\\
				&\le
				\int_{0}^{T}\left[\int_0^t\left(\int_0^{t-s} \!\left|F\left(t-s-u\right)\right|^2\left|X_u\right|\,\dd u\right)\rho\left( s\right)\dd s\right]\dd t
				=\int_{0}^{T}\left[\int_s^T\left(\left|F\right|^2\ast \left|X\right|\right)\left(t-s\right)\dd t\right]\rho\left( s\right)\dd s
				\\&=\int_{0}^{T}\left[\int_0^{T-s}\left(\left|F\right|^2\ast \left|X\right|\right)\left(t\right)\dd t\right]\rho\left( s\right)\dd s
				\le \norm{\rho}_{L^1\left(\left[0,T\right]\right)} \norm{F}^2_{L^2\left(\left[0,T\right]\right)}\norm{X}_{L^1\left(\left[0,T\right]\right)}.
			\end{split}
		\end{align*}
		Taking expectation and recalling \eqref{1.4} we have
		\begin{multline*}
			\int_{0}^{T}\mathbb{E}\left[\int_0^t\left(\int_0^t1_{\left\{\left(t-u\right)\in S\right\}}1_{\left\{s\le t-u\right\}}\left|F\left(t-s-u\right)\right|^2\rho\left( s\right)\dd s\right)\left|X_u\right|\dd u\right]\dd t
			\\\le
			 \norm{\rho}_{L^1\left(\left[0,T\right]\right)}\norm{F}^2_{L^2\left(\left[0,T\right]\right)}\mathbb{E}\left[\norm{X}_{L^1\left(\left[0,T\right]\right)}\right]<\infty.
		\end{multline*}
		This proves that there exists $N_2\subset \mathbb{R}_+$ such that  
		\begin{equation}\label{fa}
			\mathbb{E}\left[\int_0^t\left(\int_0^t1_{\left\{\left(t-u\right)\in S\right\}}1_{\left\{s\le t-u\right\}}\left|F\left(t-s-u\right)\right|^2\rho\left( s\right)\dd s\right)X_u\,\dd u\right]<\infty,\quad t\in\mathbb{R}_+\setminus N_2.
		\end{equation}
		Taking $t\in\mathbb{R}_+\setminus\left(N_1\cup N_2\right)$, thanks to \eqref{fa} and Lemma \ref{l1} (see \eqref{.1}) 
		 we can apply the stochastic Fubini's theorem in \eqref{he} to deduce that
		\begin{align*}
			&\left(\left(\rho\ast F\right)\ast \dd M^c\right)_t
			\\
			&=\int_{0}^t\!\left(\int_{0}^{t}1_{\left\{\left(t-u\right)\in S\right\}}1_{\left\{s\le t-u\right\}}F\left(t-u-s\right)\rho\left( s\right)\dd s\right)\dd M^c_u\!
			=\!\int_{0}^{t}\left(\int_{0}^{t-s}\!1_{\left\{\left(t-u\right)\in S\right\}}F\left(t-s-u\right)\dd M^c_u\right)\!\rho\left( s\right)\dd s\\
			&=\int_{0}^t\left(F\ast \dd M^c\right)_{t-s}\rho\left( s\right)\dd s
			=\left(\rho\ast\left(F\ast \dd M^c\right)\right)\left(t\right),\quad \mathbb{P}-\text{a.s.},
		\end{align*}
		and the proof is complete.
	\end{proof}
	\begin{rem}\label{rem3}
		The previous result is the analogue of \cite[Lemma $2.1$]{sergio}, where the authors are able --in the framework described in Remark \ref{rem1}-- to handle a generic signed measure of locally bounded variation ${L}$. Essentially they can do so because the convolution $F\ast \dd M$ is defined as a stochastic integral for every $t\in\mathbb{R}_+$. As a consequence, it is unique up to a  $\mathbb{P}\otimes \left|{L}\right|-$null set, being $\left|{L}\right|$ the total variation measure of ${L}$.
		\\		 In contrast with this, notice that in our setting  it is not possible to make sense of the right side of \eqref{time} for a fixed time $t>0$ when $\rho$ is replaced by ${L}$. Indeed, $F\ast \dd Z$ is only defined up to a $\mathbb{P}\otimes \dd t-$null set, therefore  the value of $\left({L}\ast\left(F\ast \dd Z\right)\right)\left(t\right)$ would depend on the modification one chooses. However,  Lemma~\ref{4.1} can be slightly extended by replacing $\rho$ in \eqref{time} with an $\mathbb{R}^{q\times p}-$valued measure  which is the sum of a locally integrable function and a point mass in $0$ (this extension can be inferred directly from \eqref{time}). We are going to need this final comment in Section \ref{sec:past}.  
	\end{rem}
	We are now ready to state  an analogue of \cite[Lemma $2.6$]{sergio}. 
	\begin{proposition}\label{p5}
		Assume that $m=d,$ and that the kernel $K\in L^2_{\emph{loc}}\left(\mathbb{R}_+;\mathbb{R}^{d\times d}\right)$ admits a resolvent of the first kind $L$\footnote{Given a kernel $K\in L^1_{\text{loc}}\left(\mathbb{R}_+; \mathbb{R}^{d\times d}\right)$, an $\mathbb{R}^{d\times d}-$valued measure $L$ is called its \emph{(measure) resolvent of the first kind} if $L\ast K=K\ast L=I$, where $I\in\mathbb{R}^{d\times d}$ is the identity matrix. $L$ does not always exists, but if it does then it is unique (cfr. \cite[Theorem $5.2$, Chapter $5$]{g}). }. Let $F\in L^2_{\emph{loc}}\left(\mathbb{R}_+;\mathbb{R}^{d\times d}\right)$ be such that $F\ast L$ is locally absolutely continuous. Then
		\begin{equation}\label{le2.6}
			\left(F\ast \dd Z\right)_t
			=\left(F\ast L\right)\left(0\right)\left(X-g_0\right)\left(t\right)+\left(\left(F\ast L\right)'\ast \left(X-g_0\right)\right)\left(t\right),\quad \text{for a.e. }t\in\mathbb{R}_+,\,\mathbb{P}-\text{a.s.}
		\end{equation}
	\end{proposition}
	\begin{proof}
		By {Lebesgue's fundamental theorem of calculus} we can write (denoting by $I$ the identity matrix in $\mathbb{R}^{d\times d}$)
		\[
		\left(F\ast L\right)\left(t\right)
		=\left(F\ast L\right)\left(0\right)+\int_{0}^{t}\left(F\ast L\right)'\left(s\right)\dd s=
		\left(F\ast L\right)\left(0\right)+\left(\left(F\ast L\right)'\ast I\right)\left(t\right),\quad t\ge0,
		\] 
		which implies, convolving with $K$, using \cite[Theorem $6.1$\,(\lowerRomannumeral{9}), Chapter $3$]{g} and a change of variables,
		\begin{equation*}
			\int_{0}^{t}F\left(s\right)\dd s=
			\left(F\ast L\right)\left(0\right)\int_{0}^{t}K\left(s\right)\dd s+\int_{0}^t\left(\left(F\ast L\right)'\ast K\right)\left(s\right)\,\dd s,\quad t\ge0.
		\end{equation*}
		We can differentiate both sides of the previous equation, as they are absolutely continuous functions, and we obtain
		\[
		F\left(t\right)=
		\left(	F\ast L\right)\left(0\right)K\left(t\right)+\left(\left(F\ast L\right)'\ast K\right)\left(t\right),\quad \text{for a.e. }t\in\mathbb{R}_+. 
		\]
		Then convolving with $\dd Z$ yields
		\begin{multline}\label{tea}
			\left(F\ast \dd Z\right)_t=
			\left(F\ast L\right)\left(0\right)\left(K\ast \dd Z\right)_t+\left(\left(\left(F\ast L\right)'\ast K\right)\ast \dd Z\right)_t\\
			=\left(F\ast L\right)\left(0\right)\left(K\ast \dd Z\right)_t+\left(\left(F\ast L\right)'\ast \left(K\ast \dd Z\right)\right)\left(t\right),\quad \mathbb{P}-\text{a.s, for a.e. }t\in\mathbb{R}_+,
		\end{multline}
		where in the first equality we use Lemma \ref{l1} (see \eqref{2.1}) and in the second Lemma \ref{4.1} with $\rho=\left(F\ast L\right)'$. The crucial point here is to pass to the trajectories. In order to do so, observe that by \eqref{2} we have
		\[
		X_t-g_0\left(t\right)=\left(K\ast \dd Z\right)_t,\quad \text{for a.e. }t\in\mathbb{R}_+,\,\mathbb{P}-\text{a.s.},
		\]
		hence $\left(\left(F\ast L\right)'\ast \left(K\ast \dd Z\right)\right)\left(t\right)=\left(\left(F\ast L\right)'\ast\left(X-g_0\right)\right)\left(t\right)$, $\mathbb{P}-$a.s., for a.e. $t\in\mathbb{R}_+$. Moreover we can consider a jointly measurable modification of the process $\left(\left(F\ast L\right)'\ast\left(X-g_0\right)\right)$ thanks to Fubini's theorem, which in turn can be applied as
		\begin{align}\label{at_last}
			\begin{split}
				&\mathbb{E}\left[\int_{0}^{T}\left(\int_{0}^{T}1_{\left\{s\le t\right\}}\left|\left(F\ast L\right)'\left(s\right)\right|\left|\left(X-g_0\right)\left(t-s\right)\right|\dd s\right)\dd t\right]
				\\&=\mathbb{E}\left[\int_{0}^{T}\left(\int_{0}^{t}\left|\left(F\ast L\right)'\left(s\right)\right|\left|\left(X-g_0\right)\left(t-s\right)\right|\dd s\right)\dd t\right]
				\\
				&\le\norm{\left(F\ast L\right)'}_{L^1\left(\left[0,T\right];\mathbb{R}^{d\times d}\right)}\left(\mathbb{E}\left[\norm{X}_{L^1\left(\left[0,T\right];\mathbb{R}^d\right)}\right]+\norm{g_0}_{L^1\left(\left[0,T\right];\mathbb{R}^d\right)}\right)<\infty,\quad T>0,
			\end{split}	
		\end{align}
		by Tonelli's theorem, Equation \eqref{1.4} and \cite[Theorem $2.2$\,(\lowerRomannumeral{1}), Chapter $2$]{g}.
		Substituting this term in \eqref{tea} and recalling once again \eqref{2} we deduce that 
		\[
		\left(F\ast \dd Z\right)_t=\left(F\ast L\right)\left(0\right)\left(X-g_0\right)\left(t\right)+\left(\left(F\ast L\right)'\ast \left(X-g_0\right)\right)\left(t\right),\quad \mathbb{P}-\text{a.s., for a.e. }t\in\mathbb{R}_+.
		\]
		This equality can be understood up to a  $\mathbb{P}\otimes \dd t-$null set because it involves only jointly measurable processes. Therefore \eqref{le2.6} holds true and  the proposition is completely proved.
	\end{proof}
	\begin{rem}\label{rem4}
		In \cite[Lemma $2.6$]{sergio} the authors require $F\ast L$ to be right--continuous and of locally bounded variation. The loss of generality in Proposition \ref{p5}, where we assume the local absolute continuity for the same function, is triggered by Lemma \ref{4.1} and Remark \ref{rem3}.
	\end{rem}
	\section{Towards the conditional Fourier--Laplace transform}\label{sec:FLT}
	In this section we are going to introduce processes $V^T=\left(V^T_t\right)_{t\in\left[0,T\right]}$ which will be used to find an \emph{ansatz} for the conditional Fourier--Laplace transform of $\left(f^\top\ast X\right)\left(T\right),\,T>0$, where $f$ is a suitable given function. 
	
	We first introduce some notation. For a $\mathbb{C}-$valued function $g\in L^1\left(\nu_k\right),\,k=0,1,\dots,m,$ we denote
	\begin{align*}
		\left\langle\eta\left(x,\dd\xi\right),g\left(\xi\right)\right\rangle&=\int_{\mathbb{R}^d}g\left(\xi\right)\nu_0\left(\dd\xi\right)
		+\sum_{k=1}^m\left(\int_{\mathbb{R}^d}g\left(\xi\right)\nu_k\left(\dd\xi\right)\right)x_k,\quad x\in E;\\
		\nu\left(g\left( \xi\right)\right)&=\begin{bmatrix}
			\int_{\mathbb{R}^d}g\left(\xi\right)\nu_1\left(\dd\xi\right)&	\int_{\mathbb{R}^d}g\left(\xi\right)\nu_2\left(\dd\xi\right)&\dots&	\int_{\mathbb{R}^d}g\left(\xi\right)\nu_m\left(\dd\xi\right)
		\end{bmatrix}^\top\in\mathbb{C}^{m}.
	\end{align*}
	Note that $\left\langle\eta\left(x,\dd \xi\right),g\left(\xi\right)\right\rangle=\int_{\mathbb{R}^d}g\left(\xi\right)\nu_0\left(\dd\xi\right)+\nu\left(g\left(\xi\right)\right)^\top x$ for every $x\in E. $ 
	In addition, we consider 
	\begin{align*}
		B&=\begin{bmatrix}b_1&b_2&\dots&b_{m}\end{bmatrix}\in\mathbb{R}^{d\times m},\\
		A\left(u\right)&=
		\begin{bmatrix}u^\top A_1\,u&u^\top A_2\,u&\dots&u^\top A_m\,u
		\end{bmatrix}^\top\in\mathbb{C}^{m},\quad u\in\mathbb{C}^d.
	\end{align*}
	Notice that $b\left(x\right)=b_0+Bx$, and  $u^\top a\left(x\right)u=u^\top A_0u+A\left(u\right)^\top x,$ for every $x\in E,\,u\in\mathbb{C}^d$.
	
	Let us take $f\in C\left(\mathbb{R_+};\mathbb{C}^m\right)$ and define $F\colon\mathbb{R}_+\times\mathbb{C}^d_-\to\mathbb{C}^m$ as follows
	\begin{equation}\label{F}
		F\left(t,u\right)= f\left(t\right)+B^\top u+\frac{1}{2}A\left(u\right)+
		\nu \left(e^{ u^\top\xi}-1- u^\top\xi \right),\quad \left(t,u\right)\in\mathbb{R}_+\times \mathbb{C}^d_-.
	\end{equation}
\begin{hyp}\label{hyp1}
		There exists a continuous, global solution $\psi\colon\mathbb{R}_+\to\mathbb{C}^d_-$ to the deterministic Riccati--Volterra equation
	\begin{equation}
		\psi\left(t\right)^\top =\int_{0}^{t}F\left(s,\psi\left(s\right)\right)^\top K\left(t-s\right)\dd s=\left(F\left(\cdot,\psi\left(\cdot\right)\right)^\top \ast K \right)\left(t\right),\quad t\ge0\label{Vp}.
	\end{equation}
\end{hyp}
	\noindent Under Hypothesis \ref{hyp1}, we introduce  the $\mathbb{C}-$valued function $\phi\colon\mathbb{R}_+\to\mathbb{C}$ given by 
	\begin{equation}\label{phi}
		\phi(t)=\int_0^t \left(\psi(s)^\top b_0 + \frac12\psi(s)^\top A_0 \psi(s) + \int_{\mathbb{R}^d}\left({\rm e}^{ \psi(s)^\top \xi}-1-\psi(s)^\top \xi \right)\nu_0(\dd\xi)\right)\dd s,\quad t\ge 0.
	\end{equation}
	
	For every $T>0$ we define the following c\`adl\`ag, adapted, $\mathbb{C}-$valued semimartingale on $\Omega\times \left[0,T\right]$
	\begin{align}
		\begin{split}
			V_t^T&=V_0^T-\int_{0}^t\left[\frac{1}{2}\psi\left(T-s\right)^\top a\left(X_s\right)\psi\left(T-s\right)+	\left\langle\eta\left(X_s,\dd \xi\right), e^{
				\psi\left(T-s\right)^\top\xi}-1- \psi\left(T-s\right)^\top\xi\right\rangle\right]\dd s\\
			&\qquad\qquad\qquad\qquad\qquad\qquad\qquad\qquad\qquad\qquad\qquad\qquad\qquad\qquad\qquad\qquad
			+\int_{0}^{t}\psi\left(T-s\right)^\top \dd \widetilde{Z}_s,\label{y1}
		\end{split}
		\\
		\begin{split}
			V_0^T\!&=\!\int_{0}^T\!\left(f\left(T-s\right)+B^\top \psi\left(T-s\right)+\frac{1}{2}A\left(\psi\left(T-s\right)\right)+\nu\left(e^{ \psi\left(T-s\right)^\top \xi}-1- \psi\left(T-s\right)^\top\xi\right)\right)^\top\!\!g_0\left(s\right)\dd s\\
			&\qquad\qquad\qquad\qquad\qquad\qquad\qquad\qquad\qquad\qquad\qquad\qquad\qquad\qquad\qquad\qquad\qquad\qquad\qquad+\phi\left(T\right).\label{y2}
		\end{split}
	\end{align}
	Observe that $V^T$ is  left--continuous in $T$ because $\psi\left(0\right)=0$ by \eqref{Vp}. This process is the natural extension of \cite[Equations $\left(4.4\right)-\left(4.5\right)$]{sergio} to the framework with jumps. Moreover, one can write
	\begin{equation}\label{y2_2}
		V_0^T=\phi\left(T\right)+\int_{0}^{T}F\left(T-s,\psi\left(T-s\right)\right)^\top g_0\left(s\right)\dd s.
	\end{equation}

	Our aim is to find, using the stochastic Fubini's theorem, an alternative expression for the random variables $V_t^T$ by means of integrals in time of the trajectories of suitable processes. \\
	In the case $b\equiv0$, we are going to use the paths of the forward process. Precisely, for a fixed $t\in\left[0,T\right]$, by \eqref{for_no b} in \ref{ap_A} we have
	\begin{equation}\label{similar}
		\mathbb{E}\left[X_s\big|\mathcal{F}_t\right]=g_0\left(s\right)+\int_{0}^{t}K\left(s-r\right)\dd \widetilde{Z}_r,\quad \mathbb{P}-\text{a.s., for a.e. }s>t.
	\end{equation}
	Hence requiring the kernel $K$ to be continuous on $\left(0,\infty\right)$, the process on the right side of the previous equation has a jointly measurable version that we denote by $\widetilde{g}_t\left(s\right),\,s>t$. Note that it makes sense to integrate in time the trajectories of such $\widetilde{g}_t\left(\cdot\right)$ since it is unique up to a $\mathbb{P}\otimes \dd t-$null set. \\
	In the case $b\neq0$ we consider the paths of a process $g_t\left(\cdot\right)$ such that
	\begin{equation}\label{G^b}
		g_t\left(s\right)= g_0\left(s\right)+\int_{0}^{t}K\left(s-r\right)\dd Z_r,\quad \mathbb{P}-\text{a.s., }s>t.
	\end{equation}
	Also in this case we assume $K$ to be continuous on $\left(0,\infty\right)$, so that $g_t\left(\cdot\right)$ can be taken jointly measurable on $\Omega\times \left(t,\infty\right)$ and  is uniquely defined up to a $\mathbb{P}\otimes \dd t-$null set. Note that when $t=0$ we have an abuse of notation, as $g_0$ represents both the initial input curve in \eqref{f} and the process just defined in \eqref{G^b}. This, however, is not an issue as these two concepts coincide $\mathbb{P}\otimes dt-$a.e. in $\Omega\times \left(0,\infty\right)$. In the following, we continue to consider $g_0$ as the initial input curve. Finally, notice that 
	\[
	g_t\left(s\right)=\mathbb{E}\left[X_s-\int_{0}^{s-t}K\left(s-t-r\right)b\left(X_{t+r}\right)\dd r\,\Big|\,\mathcal{F}_t\right],\quad \mathbb{P}-\text{a.s., for a.e. }s>t.
	\]
	For this reason  $g_t\left(\cdot\right)$ is called \emph{adjusted forward process}.
	\begin{theorem}\label{t6}
		Assume Hypothesis \ref{hyp1}. Let $K\in L^2_{\emph{loc}}\left(\mathbb{R}_+;\mathbb{R}^{m\times d}\right)$ be a continuous kernel on $\left(0,\infty\right)$ and define, for every $t\in\left[0,T\right]$,
		\begin{equation}\label{y_tilde}
			\widetilde{V}_t^T= \phi\left(T-t\right)+
			\int_{0}^{t}f\left(T-s\right)^\top X_s\,\dd s	+\int_{t}^{T}F\left(T-s,\psi\left(T-s\right)\right)^\top g_t\left(s\right)\dd s. 
		\end{equation}
		Then 
		\begin{equation}\label{thj4.3}
			V^T_t=\widetilde{V}_t^T,\quad \mathbb{P}-\text{a.s.},\, t\in\left[0,T\right].
		\end{equation}
		
		In addition, the process $\left(\exp\left\{V_t^T\right\}\right)_{t\in\left[0,T\right]}$ is a $\mathbb{C}-$valued local martingale, and if it is a true martingale then 
		\begin{equation}\label{final}
			\mathbb{E}\left[\exp\left\{\left(f^\top\ast X\right)\left(T\right)\right\}\Big|\mathcal{F}_t\right]
			=
			\exp\left\{\widetilde{V}_t^T\right\}
			,\quad \mathbb{P}-\text{a.s., }t\in\left[0,T\right].
		\end{equation}
	\end{theorem}
	\begin{proof}
		It is straightforward to check that \eqref{thj4.3} holds true for $t=0$. 

		Focusing on the case $t\in\left(0,T\right]$, we rewrite the definition of $\widetilde{V}^T_t$ in \eqref{y_tilde} as follows
		\begin{multline}\label{tt}
			\widetilde{V}^T_t=\phi\left(T-t\right)+\int_{0}^{t}f\left(T-s\right)^\top X_s\,\dd s
			\\
			+
			\int_{t}^{T}F\left(T-s,\psi\left(T-s\right)\right)^\top  g_0\left(s\right)\dd s
			+\int_{t}^{T}F\left(T-s,\psi\left(T-s\right)\right)^\top \left(g_t-g_0\right)\left(s\right)\dd s.
		\end{multline} 
		It is convenient to introduce the process  $$\widebar{g}_t\left(s\right)=\begin{cases}
			X_s,&s\le t\\
			g_t\left(s\right),&s>t
		\end{cases}.$$
		Recall that by \eqref{G^b} $g_t\left(s\right)=g_0\left(s\right)+\int_{0}^{t}K\left(s-r\right)\dd Z_r,\,\mathbb{P}-$a.s. for a.e. $s>t$, and that by \eqref{2} $X_s=g_0\left(s\right)+\int_{0}^{s}K\left(s-r\right)\dd Z_r,\,\mathbb{P}-$a.s., for a.e. $s\in\left[0,t\right]$. Therefore $\widebar{g}_t\left(\cdot\right)$ is a jointly measurable modification of the process $g_0\left(\cdot\right)+\int_{0}^{t}1_{\left\{r\le \cdot\right\}}K\left(\cdot-r\right)\dd Z_r$. Invoking the stochastic Fubini's theorem in \cite[Theorem $65$, Chapter \upperRomannumeral{4}]{protter} and recalling the Riccati--Volterra equation in \eqref{Vp}, after a suitable change of variables we obtain 
		\begin{align*}
			\begin{split}
				&\int_{0}^{T}F\left(T-s,\psi\left(T-s\right)\right)^\top  \left(\widebar{g}_t-g_0\right)\left(s\right)\dd s\\
				&=
				\int_{0}^{T}F\left(T-s,\psi\left(T-s\right)\right)^\top \left[\int_{0}^{t}1_{\left\{r\le s\right\}}K\left(s-r\right)\dd Z_r\right]\dd s
				=\int_{0}^{t}\left[\int_{r}^{T}F\left(T-s,\psi\left(T-s\right)\right)^\top K\left(s-r\right)\dd s\right]\dd Z_r\\
				&=\int_{0}^{t} \left[\int_{0}^{T-r}F\left(s,\psi\left(s\right)\right)^\top K\left(T-r-s\right)\dd s\right]\dd Z_r
				=\int_{0}^{t}\psi\left(T-r\right)^\top \dd Z_r,\quad \mathbb{P}-\text{a.s.}
			\end{split}
		\end{align*}
		Such an application is legitimate, as by the continuity of $F\left(\cdot,\psi\left(\cdot\right)\right)$ --which implies its boundedness in $\left[0,T\right]$ by a positive constant $C_T$-- and a change of variables we have (for every $k=1,\dots,m$)
		\begin{multline*}
			\int_{0}^{t}\left[\int_{0}^{T}1_{\left\{r\le s\right\}}\left|F\left(T-s,\psi\left(T-s\right)\right)\right|^2\left|K\left(s-r\right)\right|^2\dd s\right]\left|X_{k,r}\right|\dd r\\=	
			\int_{0}^{t}\left[\int_{0}^{T-r}\left|F\left(s,\psi\left(s\right)\right)\right|^2\left|K\left(T-r-s\right)\right|^2\dd s\right]\left|X_{k,r}\right|\,\dd r
			\le C_T^2\norm{K}^2_{L^2\left(\left[0,T\right];\mathbb{R}^{m\times d}\right)}\norm{X}_{L^1\left(\left[0,t\right];\mathbb{R}^m\right)},
		\end{multline*}
		so the expectation of the leftmost side is finite thanks to \eqref{1.4}. As for the drift part,
		\begin{equation*}
			\int_{0}^{t}\left(\int_{0}^{T}1_{\left\{r\le s\right\}}\left|F\left(T-s,\psi\left(T-s\right)\right)\right|^2\left|K\left(s-r\right)\right|^2\dd s\right)^\frac{1}{2}\left|X_{k,r}\right|\dd r\le C_T \norm{K}_{L^2\left(\left[0,T\right];\mathbb{R}^{m\times d}\right)}	\norm{X}_{L^1\left(\left[0,t\right];\mathbb{R}^m\right)}.
		\end{equation*}
		Going back to \eqref{tt} and  recalling the definitions of $V^T$ in \eqref{y1}--\eqref{y2_2} we obtain, $\mathbb{P}-$a.s.,
		\begin{align*}
			\widetilde{V}^T_t
			&=\phi\left(T-t\right)+\int_{0}^{t}f\left(T-s\right)^\top X_s\,\dd s+\int_{t}^{T}F\left(T-s,\psi\left(T-s\right)\right)^\top  g_0\left(s\right)\dd s\\&
			\qquad+\int_{0}^t\psi\left(T-s\right)^\top \dd Z_s
			-\int_{0}^tF\left(T-s,\psi\left(T-s\right)\right)^\top \left(X_s-g_0\left(s\right)\right)\dd s\\
			&=\phi\left(T-t\right)+\int_{0}^{T}F\left(T-s,\psi\left(T-s\right)\right)^\top g_0\left(s\right)\dd s
			+\int_{0}^t\psi\left(T-s\right)^\top \dd Z_s\\&\qquad
			-\int_{0}^{t}\left[ B^\top \psi\left(T-s\right)+\frac{1}{2}A\left(\psi\left(T-s\right)\right)+\nu\left(e^{ \psi\left(T-s\right)^\top\xi}-1-\psi\left(T-s\right)^\top\xi\right)\right]^\top X_s\,\dd s\\			
			&=\phi\left(T\right)+\int_{0}^{T}F\left(T-s,\psi\left(T-s\right)\right)^\top g_0\left(s\right)+\int_{0}^{t}\psi\left(T-s\right)^\top\dd \widetilde{Z}_s\\&\qquad
			-\int_{0}^t\left[\frac{1}{2}\psi\left(T-s\right)^\top a\left(X_s\right)\psi\left(T-s\right)
			+\left\langle\eta\left(X_s,\dd \xi\right),e^{\psi\left(T-s\right)^\top\xi}-1-\psi\left(T-s\right)^\top\xi\right\rangle\right]\dd s
			\\&=V^T_t,
		\end{align*}
		where in the second--to--last equality we use \eqref{Z}. This proofs \eqref{thj4.3}.
		
		Moving on to the next assertion, denote by $H^T=\left(H^T_t\right)_{t\in\left[0,T\right]}=\left(\exp\left\{V_t^T\right\}\right)_{t\in\left[0,T\right]}$. By  {It\^o's formula} and the dynamics in \eqref{y1} we have
		\begin{align*}
			&dH^T_{t}\\&=H^T_{t-}\left[-\left(\frac{1}{2}\psi\left(T-t\right)^\top a\left(X_t\right)\psi\left(T-t\right)+	\left\langle\eta\left(X_t,\dd \xi\right), e^{ \psi\left(T-t\right)^\top \xi}-1- \psi\left(T-t\right)^\top\xi\right\rangle\right)\dd t +\psi\left(T-t\right)^\top \dd \widetilde{Z}_t\right]\\
			&+\frac{1}{2}H^T_{t-}\psi\left(T-t\right)^\top a\left(X_t\right)\psi\left(T-t\right) \dd t
			+H^T_{t-}\int_{\mathbb{R}^d}\left(e^{ \psi\left(T-t\right)^\top \xi}-1-\psi\left(T-t\right)^\top \xi\right)\mu\left(\dd t,\dd \xi\right)\\
			&=H_{t-}^T\left[\psi\left(T-t\right)^\top dM^c_t+\int_{\mathbb{R}^d}\left(e^{\psi\left(T-t\right)^\top \xi}-1\right)\left(\mu-\nu\right)\left(\dd t,\dd \xi\right)\right],\quad H^T_0=\exp\left(V_0^T\right).
		\end{align*}
		We define $\left(N^T_t\right)_{t\in\left[0,T\right]}$ such that $\dd N^T_t= \psi\left(T-t\right)^\top \dd M^c_t+\int_{\mathbb{R}^d}\left(e^{\psi\left(T-t\right)^\top\xi}-1\right)\left(\mu-\nu\right)\left(\dd t,\dd \xi\right)$. Then $N^T$ it is a local martingale and the previous computations show that $H^T=\exp\left\{V_0^T\right\}\mathcal{E}\left(N^T\right)$ up to evanescence, where $\mathcal{E}$ denotes the Dol\'eans--Dade exponential. Therefore $H^T$ is a local martingale, as stated. Finally, in case it is a true martingale, \eqref{final} directly follows from \eqref{thj4.3}, and the proof is complete.
	\end{proof}
	\begin{rem}
		Assuming $m=d$, it is possible to find an expression for $V^T$ in terms of the \emph{true} forward process even in the case $b\neq0$, when by \eqref{forward} in \ref{ap_A}
		\begin{equation}\label{just}
			\mathbb{E}\left[X_s\big|\mathcal{F}_t\right]=\left(g_0-\left(R_B\ast g_0\right)+\left(E_B\ast b_0\right)\right)\left(s\right)+\int_{0}^{t}E_B\left(s-r\right)\dd \widetilde{Z}_r, \quad \mathbb{P}-\text{a.s., for a.e. }s>t.
		\end{equation}
		Here $R_B$ is the resolvent of the second kind \footnotemark{} of $-KB$ and $E_B=K-R_B\ast K.$ If  
		$K$ is continuous on $\left(0,\infty\right)$, then $E_B$ is continuous on the same interval, as well. 
		\footnotetext{Given $K\in L^1_{\text{loc}}\left(\mathbb{R}_+;\mathbb{R}^{d\times d}\right)$, its \emph{resolvent of the second kind} is the unique solution $R\in L^1_{\text{loc}}\left(\mathbb{R}_+;\mathbb{R}^{d\times d}\right)$ of the two equations $K\ast R=R\ast K=K-R$ (cfr. \cite[Theorem $3.1$, Chapter $2$]{g} and the subsequent definition).}
		Thus,  one can choose a jointly measurable version $f_t\left(s\right),\,s>t$, of the process on the right side of \eqref{just}, which is unique up to a $\mathbb{P}\otimes \dd t-$null set. Arguing as in \cite[Lemma $4.4$]{sergio}, we obtain the variation of constants formula
		\[
		\psi\left(t\right)^\top=\int_{0}^{t}\left[f\left(s\right)+\frac{1}{2}A\left(\psi\left(s\right)\right)+\nu\left(e^{\psi\left(s\right)^\top \xi}-1-\psi\left(s\right)^\top\xi\right)\right]^\top E_B\left(t-s\right)\dd s,\quad t\ge0,
		\]
		which combined with the strategy in the proof of Theorem \ref{t6} leads to
		\begin{multline*}
			{V}_t^T=
			\int_{0}^{t}f\left(T-s\right)^\top X_s\,\dd s+\int_{t}^{T}\bigg[\left(
			F\left(T-s,\psi\left(T-s\right)\right)-B^\top\psi\left(T-s\right)
			\right)
			^\top f_t\left(s\right)\\
			+\frac{1}{2}\psi\left(T-s\right)^\top A_0\psi\left(T-s\right)+\int_{\mathbb{R}^d}\left(e^{ \psi\left(T-s\right)^\top \xi}-1- \psi\left(T-s\right)^\top\xi\right)\nu_0\left(\dd \xi\right)\bigg]
			\,\dd s,\quad \mathbb{P}-\text{a.s.}
		\end{multline*}
		However in the framework of jumps it is preferable to work with the adjusted forward process, because --as will become clear in the next section-- certain properties can be assumed for the kernel $K$, but they can be neither required (i.e. it would not be a reasonable hypothesis) nor inferred for $E_B$.
	\end{rem}
	\section{An expression for $V^T$ affine in the past trajectory of $X$}\label{sec:past}
	In this section we consider $m=d$ and aim to find an alternative formula for $V^T$ which is affine in the past trajectory of $X$. This new expression can be used to prove the martingale property of the complex--valued process $\exp\left\{V^T\right\}$ in particular cases (see Section \ref{sec:example}). Due to the lack of regularity of the trajectories of both $X$ and the stochastic convolution in $\dd Z$,  we are going to require mild, additional conditions on the kernel $K$, in particular on the shifted kernels $\Delta_hK$ for $h>0$.
	
	We start with a preliminary result providing an alternative expression for the adjusted forward process $g_t\left(\cdot\right)$.
	\begin{lemma}\label{l7}
		Assume that $K\in L^2_{\emph{loc}}\left(\mathbb{R}_+;\mathbb{R}^{d\times d}\right)$ is continuous on $\left(0,\infty\right)$ and  that it admits a resolvent of the first kind $L$ with no point masses in $\left(0,\infty\right).$ In addition, suppose that for every $h>0$ the shifted kernel $\Delta_hK$ is differentiable, with derivative $\left(\Delta_hK\right)'\in C\left(\mathbb{R}_+;\mathbb{R}^{d\times d}\right)$. Then, for every $T>0$, for every $t\in\left[0,T\right)$  
		\begin{equation}\label{conde}
			g_t\left(T\right)=g_0\left(T\right)
			+K\left(T-t\right)Z_t
			+\left(\left(\left(\Delta_{T-t}K\right)'\ast L\right)\ast \left(X-g_0\right)\right)\left(t\right),\quad \mathbb{P}-\text{a.s.}
		\end{equation}
	\end{lemma}
	\begin{proof}
		Let us fix $h>0$. We first show that the stochastic convolution $\Delta_hK\ast \dd Z$ has a c\`adl\`ag version. Indeed, for every $t\in\mathbb{R}_+$, $\left(\Delta_hK\ast \dd Z\right)_t=\int_{0}^{t}f_{t,h}\left(s\right)\dd Z_s,\,\mathbb{P}-$a.s., with $f_{t,h}\left(s\right)= K\left(t+h-s\right),\,s\in\left[0,t\right]$. Integration by parts 
		yields
		\begin{multline*}
			\int_{0}^{t}f_{t,h}\left(s\right)\dd Z_s=
			f_{t,h}\left(t\right)Z_t-f_{t,h}\left(0\right)Z_0-\int_{0}^{t}\left({f_{t,h}}\right)'\left(s\right)Z_{s-}\,\dd s\\=\!K\left(h\right)Z_t+\!\int_{0}^{t}\left(\Delta_hK\right)'\left(t-s\right)Z_{s}\,\dd s,\quad \mathbb{P}-\text{a.s.},
		\end{multline*} 
		where we also note that $Z_{t-}=Z_t$ for a.e. $t>0,\,\mathbb{P}-$a.s. Since the rightmost side of the previous equality is a c\`adl\`ag process we obtain the desired claim. Hence in what follows we consider $\Delta_hK\ast \dd Z$ to be right--continuous. In particular, the process $\left(\Delta_hK-K\left(h\right)\right)\ast \dd Z$ is continuous.
		
		Thanks to the assumptions on the kernel, we apply \cite[Corollary $7.3$, Chapter $3$]{g} to claim that the function $\left(\Delta_hK-K\left(h\right)\right)\ast L$ is locally absolutely continuous in $\mathbb{R}_+$, with 
		\[
		\left(\left(\Delta_hK-K\left(h\right)\right)\ast L\right)'=\left(\Delta_hK\right)'\ast L,\quad \text{a.e. in }\mathbb{R}_+.
		\]
		In particular, the function $\left(\Delta_h K\right)'\ast L\in C\left(\mathbb{R}_+;\mathbb{R}^{d\times d}\right)$ by \cite[Corollary $6.2$\,(\lowerRomannumeral{3}), Chapter $3$]{g}, the absence of point masses of $L$ in $\left(0,\infty\right)$ and the continuity of $\left(\Delta_h K\right)'$. Therefore we invoke Proposition \ref{p5} to obtain
		\begin{align*}
			\left(\left(\Delta_hK-K\left(h\right)\right)\ast \dd Z\right)_t
			&=\left(\left(\Delta_hK-K\left(h\right)\right)\ast L\right)\left(0\right)\left(X-g_0\right)\left(t\right)+\left(\left(\left(\Delta_hK\right)'\ast L\right)\ast \left(X-g_0\right)\right)\left(t\right)\\
			&=\left(\left(\left(\Delta_hK\right)'\ast L\right)\ast \left(X-g_0\right)\right)\left(t\right),\quad \text{ for a.e. }t\in\mathbb{R}_+,\,\mathbb{P}-\text{a.s.}
		\end{align*}
		Note that the last equality involves continuous processes, so it is indeed true for every $t\ge0$ up to a $\mathbb{P}-$null set. Thus,
		\begin{equation}\label{d}
			\left(\Delta_hK\ast \dd Z\right)_t=
			K\left(h\right)Z_{t}
			+\left(\left(\left(\Delta_hK\right)'\ast L\right)\ast \left(X-g_0\right)\right)\left(t\right),\quad t\ge0,\,\mathbb{P}-\text{a.s.}
		\end{equation}
		At this point, let us take $t<T$ and recall that (by \eqref{G^b})
		\begin{multline*}
			g_t\left(T\right)=
			g_0\left(T\right)+\int_{0}^{t}K\left(T-s\right)\dd Z_s=g_0\left(T\right)+\int_{0}^{t}\left(\Delta_{T-t} K\right)\left(t-s\right)\dd Z_s\\
			=g_0\left(T\right)+\left(\Delta_{T-t}K\ast \dd Z\right)_t,\quad \mathbb{P}-\text{a.s.}
		\end{multline*}
		It suffices to take $h=T-t$ in \eqref{d} to deduce that
		\[
		\left(\Delta_{T-t}K\ast \dd Z\right)_t=K\left(T-t\right)Z_t+\left(\left(\left(\Delta_{T-t}K\right)'\ast L\right)\ast \left(X-g_0\right)\right)\left(t\right),\quad \mathbb{P}-\text{a.s.}.
		\] 
		Hence, combining the two previous equations, we conclude
		\begin{equation*}
			g_t\left(T\right)=g_0\left(T\right)
			+K\left(T-t\right)Z_t
			+\left(\left(\left(\Delta_{T-t}K\right)'\ast L\right)\ast \left(X-g_0\right)\right)\left(t\right),\quad \mathbb{P}-\text{a.s.},
		\end{equation*}
		completing the proof.
	\end{proof}
	Let us fix a generic $T>0$. 
	By Equation \eqref{conde} we can write, for every $t\in\left[0,T\right)$,
	\begin{equation}\label{pain}
		g_t\left(s\right)=g_0\left(s\right)+K\left(s-t\right)Z_t
		+\left(\left(\left(\Delta_{s-t}K\right)'\ast L\right)\ast \left(X-g_0\right)\right)\left(t\right),\quad \mathbb{P}-\text{a.s.},\,s\in\left(t,T\right).
	\end{equation}
	Intuitively speaking, we want to plug this expression in \eqref{y_tilde}, so that we end up with an alternative formulation for $V^T_t$ which is an affine function on the past trajectory $\left\{X_s, s\le t\right\}$.  This is done in the next theorem, which extends \cite[Theorem $4.5$]{sergio} under further conditions on the kernel $K$. These addtional assumptions hold for instance in the one--dimensional case if $K$ is completely monotone (recall that a function $f$ is called completely monotone on $(0,\infty)$ if it is
infinitely differentiable there with $(-1)^kf^{(k)}(t) \ge 0$ for all $t > 0$ and $k = 0, 1, \ldots$).
	\begin{theorem}\label{main}
		Assume that $K\in L^2_\emph{loc}\left(\mathbb{R}_+;\mathbb{R}^{d\times d}\right)$ is continuous on $\left(0,\infty\right)$ and  that it admits a resolvent of the first kind $L$ with no point masses in $\left(0,\infty\right)$.  In addition, suppose that for every $h>0$ the shifted kernel $\Delta_hK$ is differentiable, with $\left(\Delta_hK\right)'$ continuous on $\mathbb{R}_+$. Under Hypothesis \ref{hyp1}, if the total variation bound 
		\begin{equation}\label{4.15}
			\sup_{h\in\left(0,\widebar{T}\right]}\norm{\Delta_hK \ast L}_{\text{TV}\left(\left[0,\widebar{T}\right]\right)}<\infty,\quad \text{for all }\widebar{T}>0,
		\end{equation}
		holds, then for every $h>0$ the $\mathbb{C}^d-$valued function 
		\begin{equation}\label{pi}
			\pi_h\left(r\right)= \left(F\left(\cdot,\psi\left(\cdot\right)\right)^\top \ast\left(\left(\Delta_\cdot K\right)'\ast L\right)\left(r\right)\right)\left(h\right)^\top 
		\end{equation}
		is well--defined for a.e. $r\in\mathbb{R}_+$ and belongs to $L_\emph{loc}^1\left(\mathbb{R}_+;\mathbb{C}^d\right)$. Moreover, $\mathbb{P}-$a.s., for a.e. $t\in \left(0,T\right)$,
		\begin{multline}\label{past}
			V^T_t=\phi\left(T-t\right)+\int_{0}^{t}f\left(T-s\right)^\top X_s\,\dd s+\int_{0}^{T-t}F\left(s,\psi\left(s\right)\right)^\top g_0\left(T-s\right)\dd s\\
			+\psi\left(T-t\right)^\top Z_t+\left({\pi_{T-t}}^\top \ast \left(X-g_0\right) \right)\left(t\right),
		\end{multline}
		where $\phi$ is defined in \eqref{phi}.
	\end{theorem}
	\begin{proof}
		Fix $h>0$; expanding the notation in \eqref{pi} for $\pi_h$ we have
		\begin{align*}
			\pi_h\left(r\right)^\top&=\int_0^hF\left(s,\psi\left(s\right)\right)^\top\left[\int_{0}^r \left(\Delta_{h-s} K\right)'\left(r-u\right)L\left(\dd u\right)\right]\dd s.
		\end{align*}
		In order to see that it is well--defined a.e. on $\mathbb{R}_+$ and belongs to $L^1_{\text{loc}}\left(\mathbb{R}_+;\mathbb{C}^d\right)$, first note that for every positive $s$, the  continuity of $\left(\Delta_sK\right)'$ and the absence of point masses for $L$ in $\left(0,\infty\right)$ allow to apply \cite[Corollary $6.2$\,(\lowerRomannumeral{3}), Chapter $3$]{g}, which ensures the continuity on $\mathbb{R}_+$ of $\left(\Delta_sK\right)'\ast L$. As a consequence, we can define the $\mathbb{C}^d-$valued measurable function
		\[
		\left[F\left(s,\psi\left(s\right)\right)^\top\left(\left(\Delta_{h-s}K\right)'\ast L\right)\left(r\right)\right]^\top ,\quad \left(s,r\right)\in\left(0,h\right)\times\mathbb{R}_+.
		\]
		Recalling the previous proof, we see that $\left(\Delta_hK\right)'\ast L$ is, almost everywhere, the derivative of the locally absolutely continuous function $\left(\Delta_hK-K\left(h\right)\right)\ast L$. The continuity of $F\left(\cdot,\psi\left(\cdot\right)\right)$ (which implies its boundedness by a constant $C_h>0$ on $\left[0,h\right]$) coupled with Condition \eqref{4.15}, Tonelli's theorem and \cite[Theorem $6.1$\,(\lowerRomannumeral{5}), Chapter $3$]{g} yields, for a generic $\widebar{T}>h$,
		\begin{align*}\begin{split}
				&\int_{0}^{\widebar{T}}\left[\int_{0}^{h}\left|F\left(s,\psi\left(s\right)\right)\right|\left|\left(\left(\Delta_{h-s}K\right)'\ast L\right)\left(r\right)\right|\dd s\right]\dd r
				=\int_{0}^h\left|F\left(s,\psi\left(s\right)\right)\right|\left[\int_{0}^{\widebar{T}}\left|\left(\left(\Delta_{h-s}K\right)'\ast L\right)\left(r\right)\right|\dd r\right]\dd s\\
				&\le d^2
				\int_{0}^h\left|F\left(s,\psi\left(s\right)\right)\right|\left[\norm{\left(\Delta_{h-s}K-K\left(h-s\right)\right)\ast L}_{\text{TV}\left(\left[0,\widebar{T}\right]\right)}\right]\dd s
				\\
				&\le d^2 \left[\int_{0}^{h}\left|F\left(s,\psi\left(s\right)\right)\right|\norm{\Delta_{h-s}K\ast L}_{\text{TV}\left(\left[0,\widebar{T}\right]\right)} \dd s
				+\left|L\right|\left(\left[0,\widebar{T}\right]\right)\int_{0}^{h}\left|F\left(s,\psi\left(s\right)\right)\right|\left|K\left(h-s\right)\right|\dd s\right]
				\\ &\le d^2
				\left[\sup_{s\in\left(0,\widebar{T}\right]}\norm{\Delta_{s}K\ast L}_{\text{TV}\left(\left[0,\widebar{T}\right]\right)}C_h h+
				\left|L\right|\left(\left[0,\widebar{T}\right]\right)\int_{0}^{h}\left|F\left(s,\psi\left(s\right)\right)\right|\left|K\left(h-s\right)\right|\dd s\right]<\infty.
			\end{split}
		\end{align*}
		Hence the conclusion on $\pi_h$ follows. Furthermore, by {Lebesgue's fundamental theorem of calculus}, the   $\mathbb{C}^d-$valued function $\Pi_h\left(r\right)=\int_{0}^{r}\pi_h\left(u\right)\dd u,\,r\in\mathbb{R}_+$, is locally absolutely continuous on $\mathbb{R}_+$, with $\Pi_h'=\pi_h$ a.e. Using Fubini's theorem we can obtain the following explicit expression for such $\Pi_h$
		\begin{equation}\label{PI}
			\Pi_h\left(r\right)^\top=\int_{0}^{h}F\left(s,\psi\left(s\right)\right)^\top \left(\left(\Delta_{h-s}K-K\left(h-s\right)\right)\ast L\right)\left(r\right)\dd s,\quad r\in\mathbb{R}_+.
		\end{equation}
		At this point we observe that for every  function $g\in L^1_\text{loc}\left(\mathbb{R}_+;\mathbb{R}^d\right)$ we have, reasoning as before and using the boundedness of $F\left(\cdot,\psi\left(\cdot\right)\right)$ on $\left[0,T\right]$ by a positive constant $C_T$,
		\begin{align}\label{omg}
			\begin{split}
				&\int_{0}^{T}\left[\int_{0}^{t}\left|g\left(t-u\right)\right|\left(\int_{0}^{T-t}\left|F\left(s,\psi\left(s\right)\right)\right|\left|\left(\left(\Delta_{T-t-s}K\right)'\ast L\right)\left(u\right)\right|\dd s\right)\dd u\right]\dd t\\
				&=\int_{0}^{T}\left[\int_{0}^{t}\left|g\left(t-u\right)\right|\left(\int_{0}^{T-t}\left|F\left(T-t-s,\psi\left(T-t-s\right)\right)\right|\left|\left(\left(\Delta_{s}K\right)'\ast L\right)\left(u\right)\right|\dd s\right)\dd u\right]\dd t\\
				&=
				\int_{0}^{T}\left[\int_{0}^{T-t}\left|F\left(T-t-s,\psi\left(T-t-s\right)\right)\right|\left(\int_{0}^{t}\left|g\left(t-u\right)\right|\left|\left(\left(\Delta_{s}K\right)'\ast L\right)\left(u\right)\right|\dd u\right)\dd s\right]\dd t\\
				&=
				\int_{0}^{T}\left[\int_{0}^{T-s}\left|F\left(T-s-t,\psi\left(T-s-t\right)\right)\right|\left(\int_{0}^{t}\left|g\left(t-u\right)\right|\left|\left(\left(\Delta_{s}K\right)'\ast L\right)\left(u\right)\right|\dd u\right)\dd t\right]\dd s\\
				&\le C_Td^2\norm{g}_{L^1\left(\left[0,T\right];\mathbb{R}^d\right)}\int_{0}^{T}\norm{\left(\Delta_{s}K-K\left(s\right)\right)\ast L}_{\text{TV}\left(\left[0,T\right]\right)}\dd s
				\\&\le C_Td^2\norm{g}_{L^1\left(\left[0,T\right];\mathbb{R}^d\right)}\left[T\sup_{s\in\left(0,T\right]}\norm{\Delta_sK\ast L}_{\text{TV}\left(\left[0,T\right]\right)}+\left|L\right|\left(\left[0,T\right]\right)\int_{0}^{T}\left|K\left(s\right)\right|\dd s\right]<\infty,
			\end{split}
		\end{align}
		where we apply Tonelli's theorem, together with \cite[Theorem $2.2$\,(\lowerRomannumeral{1}), Chapter $2$]{g} and a change of variables. Consequently, for almost every $t\in\left(0,T\right)$ we can apply Fubini's theorem to obtain
		\begin{equation}\label{pi'}
			\int_{0}^{t}\pi_{T-t}\left(u\right)^\top g\left(t-u\right)\dd u
			=\int_{0}^{T-t}F\left(s,\psi\left(s\right)\right)^\top \left[\int_0^{t}\left(\left(\Delta_{T-t-s}K\right)'\ast L\right)\left(u\right)g\left(t-u\right)\dd u\right]\dd s.
		\end{equation}
		Computations analogous to those in \eqref{omg} (with $g$ [resp., $\left|F\left(\cdot,\psi	\left(\cdot\right)\right)\right|$] substituted by $X-g_0$ [resp., $1$]) let us conclude by Fubini's theorem and Equation \eqref{1.4} that for a.e. $t\in\left(0,T\right)$ there is a jointly measurable modification of the process $\left(\left(\left(\Delta_{\cdot-t}K\right)'\ast L\right)\ast \left(X-g_0\right)\right)\left(t\right)$ on $\Omega\times\left(t,T\right)$. Therefore we interpret \eqref{pain} trajectoriwise, namely the equality holds almost everywhere in $\left(t,T\right)$ up to a $\mathbb{P}-$null set.
		
		Now we focus on $\widetilde{V}_t^T$. Combining \eqref{y_tilde} with what we have just said, a suitable change of variables yields 
		\begin{align*}
			\widetilde{V}_t^T&=\phi\left(T-t\right)+\int_{0}^{t}f\left(T-s\right)^\top X_s\,\dd s
			+\int_{0}^{T-t}F\left(s,\psi\left(s\right)\right)^\top g_t\left(T-s\right)\dd s\notag\\
			&=\left\{\phi\left(T-t\right)+\int_{0}^{t}f\left(T-s\right)^\top X_s\,\dd s+\int_{0}^{T-t}
			F\left(s,\psi\left(s\right)\right)^\top g_0\left(T-s\right)
			\dd s\right\}\notag\\\notag
			&\quad+\left\{\left(\int_{0}^{T-t}F\left(s,\psi\left(s\right)\right)^\top K\left(T-t-s\right)\dd s\right)Z_t\right\}\\&\quad\quad + \left\{\int_{0}^{T-t}F\left(s,\psi\left(s\right)\right)^\top \left(\left(\left(\Delta_{T-t-s}K\right)'\ast L\right)\ast \left(X-g_0\right)\right)\left(t\right)\dd s\right\} \notag
			\\&= {\bf \upperRomannumeral{1}}_t+{\bf \upperRomannumeral{2}}_t+{\bf \upperRomannumeral{3}}_t,\quad \mathbb{P}-\text{a.s., }t\in\left(0,T\right).	
		\end{align*}
		The idea is to analyze separately the addends that we have singled out in the previous computations.
		Note that ${\bf \upperRomannumeral{3}}_t$ is finite because  $\widetilde{V}_t^T, {\bf \upperRomannumeral{1}}_t,{\bf \upperRomannumeral{2}}_t$ are so,  and that we can consider a jointly measurable modification of this process in $\Omega\times \left(0,T\right)$, again by Fubini's theorem and Equation \eqref{1.4} (see \eqref{omg}). 		
		Taking into account \eqref{thj4.3} we have
		\begin{equation}\label{path}
			V_t^T=	{\bf \upperRomannumeral{1}}_t+	{\bf \upperRomannumeral{2}}_t+	{\bf \upperRomannumeral{3}}_t,\quad  \text{for a.e. }t\in\left(0,T\right),\,\mathbb{P}-\text{a.s.},
		\end{equation}
		where the equality can be understood pathwise as it involves jointly measurable processes. \\
		Regarding ${\bf \upperRomannumeral{2}}_t$, since $\psi$ solves the Riccati--Volterra equation in \eqref{Vp} we have
		\[
		{\bf \upperRomannumeral{2}}_t=\psi\left(T-t\right)^\top Z_t.
		\]
		As for ${\bf \upperRomannumeral{3}}_t$, by \eqref{pi'} we have 
		\begin{equation*}
			{\bf \upperRomannumeral{3}}_t
			=\int_{0}^{t}\pi_{T-t}\left(u\right)^\top \left(X-g_0\right)\left(t-u\right)\dd u=\left({\pi_{T-t}}^\top \ast \left(X-g_0\right)\right)\left(t\right),\quad \text{for a.e. $t\in\left(0,T\right), \,\mathbb{P}-$a.s.}
		\end{equation*}
		Substituting the two previous equations in \eqref{path} we conclude
		\begin{multline*}
			{V}_t^T=\phi\left(T-t\right)+\int_{0}^{t}f\left(T-s\right)^\top X_s\,\dd s+\int_{0}^{T-t}
			F\left(s,\psi\left(s\right)\right)^\top g_0\left(T-s\right)
			\dd s\\
			+\psi\left(T-t\right)^\top Z_t+\left( {\pi_{T-t}}^\top  \ast \left(X-g_0\right)\right)\left(t\right)
			,
		\end{multline*}
		for almost every  $t\in\left(0,T\right),\, \mathbb{P}-$a.s. 
		The proof is now complete.
	\end{proof}
	If the resolvent of the first kind $L$ is the sum of a locally integrable function and a point mass in $0$, then recalling \eqref{2} we can apply Lemma \ref{4.1} (see also the final comment in Remark \ref{rem3}) and argue as in \eqref{at_last} to see that $Z_{t}=\left(L\ast\left(X-g_0\right) \right)\left(t\right),$ for a.e. $t>0$, $\mathbb{P}-$a.s.  In addition,  for every $h>0$  we define the $\mathbb{C}^d-$valued function 
	\begin{equation}\label{PI_tilde}
		\widetilde{\Pi}_h\left(r\right)^\top=\Pi_h\left(r\right)^\top+\psi\left(h\right)^\top L\left(\left\{0\right\}\right)+\psi\left(h\right)^\top L\left(\left(0,r\right]\right)=\int_{0}^{h}F\left(s,\psi\left(s\right)\right)^\top\left( \Delta_{h-s} K\ast L\right)\left(r\right)\dd s,\quad r\in\mathbb{R}_+,
	\end{equation}
	where the second equality is due to \eqref{PI}. Note that $\widetilde{\Pi}_h$ is locally absolutely continuous on $\mathbb{R}_+$, and that
	\begin{multline*}
		\left( {\pi_{T-t}}^\top \ast\left(X-g_0\right) \right)\left(t\right)=
		\left( \dd{\Pi_{T-t}}^\top \ast\left(X-g_0\right)\right)\left(t\right)\\
		=\left({\dd \widetilde{{\Pi}}_{T-t}}^\top\ast \left(X-g_0\right)\right)\left(t\right)-\psi\left(T-t\right)^\top \left(L\ast\left(X-g_0\right) \right)\left(t\right)+\psi\left(T-t\right)^\top L\left(\left\{0\right\}\right)\left(X-g_0\right)\left(t\right), 
	\end{multline*}
	holding true for a.e. $t\in\left(0,T\right),\,\mathbb{P}-$a.s.
	Substituting in \eqref{past} we immediately deduce the following result.
	\begin{corollary}\label{c9}
		Under the same hypotheses of Theorem \ref{main}, if the resolvent of the first kind $L$ is the sum of  a locally integrable function and a point mass in $0$, then $\mathbb{P}-$a.s., for a.e. $t\in\left(0,T\right)$  
		\begin{multline}\label{past1}
			V^T_t=\phi\left(T-t\right)+\int_{0}^{t}f\left(T-s\right)^\top X_s\,\dd s+\int_{0}^{T-t}
			F\left(s,\psi\left(s\right)\right)^\top g_0\left(T-s\right)
			\dd s\\+
			\psi\left(T-t\right)^\top L\left(\left\{0\right\}\right)\left(X-g_0\right)\left(t\right)
			-\left({\dd\widetilde{\Pi}_{T-t}}^\top \ast g_0\right)\left(t\right)+\left({\dd\widetilde{\Pi}_{T-t}}^\top \ast X\right)\left(t\right).
		\end{multline}
	\end{corollary}
	\section{The ${1}-$dimensional Volterra square root diffusion process with jumps }\label{sec:example}
	In this section we discuss a one--dimensional example ($m=d=1$) where not only are we able to infer the assumptions made in the previous arguments, such as the existence of solutions to the stochastic Volterra equation \eqref{f} and the Riccati--Volterra equation \eqref{Vp} (i.e., Hypothesis \ref{hyp1}), but also we can prove the martingale property of the process $\exp\left\{V^T\right\}.$ 
	In order to develop the theory  we need to require more properties for the kernel $K$. In particular, we consider a hypothesis which is standard in the theory of stochastic Volterra equations, that is (see \cite[Condition $\left(2.10\right)$]{edu}, and also  \cite[Condition $\left(3.4\right)$]{sergio} and \cite[Assumption B.$2$]{ee})
	\begin{Condition}\label{c1}
		The kernel $K$ is nonnegative, nonincreasing, not identically zero and continuously differentiable on $\left(0,\infty\right)$, and its resolvent of the first kind $L$ is nonnegative and nonincreasing, i.e., $s\mapsto L\left[s,s+t\right]$ is nonincreasing for every $t\ge0$. 
	\end{Condition}
	We focus on the following stochastic Volterra equation of convolution type
	\begin{equation}\label{f1}
		X=g_0+\left(K\ast \dd Z\right),\quad \mathbb{P}\otimes \dd t-\text{a.e.},
	\end{equation}
	where $Z$ is a real--valued semimartingale with differential characteristics (with respect to $h\left(\xi\right)=\xi,\,\xi\in\mathbb{R}$) given by $\left(b\left(X_t\right),a\left(X_t\right),\eta\left(X_t,\dd\xi\right)\right),\,t\ge0$, with 
	\[
	b\left(x\right)=bx,\qquad a\left(x\right)=c\,x,\qquad \eta\left(x,\dd\xi\right)=x\nu\left(\dd\xi\right),\quad x\ge0.
	\]
	Here $b\in\mathbb{R}, \,c\ge0$ and $\nu$ is a nonnegative measure on $\mathbb{R}_+$ such that $\int_{\mathbb{R}_+}\left|\xi\right|^2\nu\left(\dd\xi\right)<\infty.$
	The function $g_0\colon\mathbb{R}_+\to \mathbb{R}$ is an admissible input curve in either one of the following two forms
	\begin{enumerate}[label=\roman*.]
		\item $g_0$ is continuous and non--decreasing, with $g\left(0\right)=0$;
		\item $g_0\left(t\right)=x_0+\int_{0}^{t}K\left(t-s\right)\theta\left(s\right)\dd s,\,t\ge0$, where $x_0\ge 0$ and  $\theta\colon\mathbb{R}_+\to\mathbb{R}_+$ is locally bounded.
	\end{enumerate}
	Notice that \eqref{f1} describes a $1-$dimensional Volterra square root diffusion. 
	In this framework, we can invoke \cite[Theorem $2.13$]{edu} to claim the existence of a weak, predictable solution $X=\left(X_t\right)_{t\ge 0}$ of \eqref{f1} with trajectories in $L^1_\text{loc}\left(\mathbb{R}_+\right)$ such that $X\ge0,\,\mathbb{P}\otimes \dd t-$a.e. Actually, if $g_0\in L^2_\text{loc}\left(\mathbb{R}_+\right)$, the paths of this solution $X$ are in $L^2_\text{loc}\left(\mathbb{R}_+\right),\,\mathbb{P}-$a.s., as the next result shows. 
	\begin{lemma}\label{L2}
		Suppose that $g_0\in L^2_\emph{loc}\left(\mathbb{R}_+\right)$ and let $X$ be a solution of \eqref{f1} with trajectories in $L^1_\emph{loc}\left(\mathbb{R}_+\right)$ such that $X\ge 0,\,\mathbb{P}\otimes \dd t-$a.e. Then, for every $T>0$, $\mathbb{E}\left[\left(\int_{0}^{T}\left|X_t\right|^2\dd t\right)^{1/2}\right]<\infty$. 
	\end{lemma}
	\begin{proof}
		The convolution equation \eqref{f1} let us write, $\mathbb{P}-$a.s.,
		\begin{equation*}
			\left|X_t\right|^2\le 4\left(\left|g_0\left(t\right)\right|^2+\left|b\right|^2\left|\left(K\ast X\right)\left(t\right)\right|^2+\left|\left(K\ast \dd M^c\right)_t\right|^2+\left|\left(K\ast \dd M^d\right)_t\right|^2\right),\quad \text{for a.e. } t\ge0.
		\end{equation*}
		Integrating over the interval $\left(0,T\right), T>0$, we have
		\begin{multline*}
			\left(\int_{0}^{T}\left|X_t\right|^2\dd t\right)^{\frac{1}{2}}\le 2\bigg(2+\norm{g_0}_{L^2\left(\left[0,T\right]\right)}+\left|b\right|\norm{K\ast X}_{L^2\left(\left[0,T\right]\right)}\\+\int_{0}^{T}\left|\left(K\ast \dd M^c\right)_t\right|^2\dd t+\int_{0}^{T}\left|\left(K\ast \dd M^d\right)_t\right|^2\dd t\bigg),\quad \mathbb{P}-\text{a.s.,}
		\end{multline*}
		where we also use that $\sqrt{x}\le 1+x,\,x\in\mathbb{R}_+.$ By \cite[Theorem $2.2$\,(\lowerRomannumeral{1}), Chapter $2$]{g} we have $\norm {K\ast X}_{L^2\left(\left[0,T\right]\right)}\le \norm{K}_{L^2\left(\left[0,T\right]\right)} \norm{ X}_{L^1\left(\left[0,T\right]\right)}$, hence taking expectation in the previous inequality we obtain, using Tonelli's theorem,
		\begin{multline}\label{close}
			\mathbb{E}\left[\left(\int_{0}^{T}\left|X_t\right|^2\dd t\right)^{\frac{1}{2}}\right]\le 
			2\bigg(2+\norm{g_0}_{L^2\left(\left[0,T\right]\right)}+\left|b\right|\norm{K}_{L^2\left(\left[0,T\right]\right)} \mathbb{E}\left[\norm{ X}_{L^1\left(\left[0,T\right]\right)}\right]\\
			+
			\int_{0}^{T}\mathbb{E}\left[\left|\left(K\ast \dd M^c\right)_t\right|^2\right]\dd t
			+
			\int_{0}^{T}\mathbb{E}\left[\left|\left(K\ast \dd M^d\right)_t\right|^2\right]\dd t\bigg).
		\end{multline}
		Recall that $\left(K\ast \dd M^c\right)_t=\int_{0}^{t}K\left(t-s\right)\dd M^c_s,\, \mathbb{P}-$a.s. for a.e. $t\ge 0$; therefore we use the Burkholder--Davis--Gundy inequality and the Young's type inequality in \cite[Lemma A$.1$]{ACLP}  to write (always bearing in mind Tonelli's theorem)
		\begin{multline*}
			\int_{0}^{T}\mathbb{E}\left[\left|\left(K\ast \dd M^c\right)_t\right|^2\right]\dd t\le c\cdot c_1\mathbb{E}\left[\int_{0}^T \left(\int_{0}^{t}\left|K\left(t-s\right)\right|^2X_s\,\dd s\right)\dd t\right]
			\\\le c\cdot  c_1\norm{K}^2_{L^2\left(\left[0,T\right]\right)}\mathbb{E}\left[\norm{X}_{L^1\left(\left[0,T\right]\right)}\right]
			,\quad \text{for some }c_1>0.
		\end{multline*}
		Analogously, we invoke \cite[Theorem $3.2$]{RM} to assert
		\[
		\int_{0}^{T}\mathbb{E}\left[\left|\left(K\ast \dd M^d\right)_t\right|^2\right]\dd t\le 2 \left(\int_{\mathbb{R}_+}\left|\xi\right|^2\nu\left(\dd \xi\right)\right)c_2\norm{K}^2_{L^2\left(\left[0,T\right]\right)}\mathbb{E}\left[\norm{X}_{L^1\left(\left[0,T\right]\right)}\right],\quad \text{for some }c_2>0.
		\]
		Now substituting the previous two bounds in \eqref{close}  we see that the right side is finite by \eqref{1.4}. This concludes the proof.
	\end{proof}
	 Continuing the example of this section, let $f\in C\left(\mathbb{R}_+;\mathbb{C}_-\right)$. In addition to Eq.~\eqref{Vp}, we consider the deterministic Riccati--Volterra equation 
	\begin{equation}\label{Vp_real}
		\widebar{\psi}\left(t\right)=\int_{0}^{t}K\left(t-s\right)\widebar{F}\left(s,\widebar{\psi}\left(s\right)\right)\dd s,\quad t\ge0,
	\end{equation}
where $\mathfrak{R}f\colon\mathbb{R}_+\to\mathbb{R}_-$ denotes the real part of $f$ and  $\widebar{F}\colon\mathbb{R}_+\times\mathbb{C}_-\to\mathbb{R}$ is defined as follows
	\begin{equation}\label{Fbar}
	\widebar{F}\left(t,u\right)= \mathfrak{R}{f}\left(t\right)+bu+\frac{c}{2}u^2+\int_{\mathbb{R}_+}\left(e^{u\xi}-1-u\xi\right)\nu\left(\dd\xi\right),\quad \left(t,u\right)\in\mathbb{R}_+\times\mathbb{C}_-.
\end{equation}
The next theorem shows the existence of global solutions to \eqref{Vp} and \eqref{Vp_real} (in particular, Hypothesis \ref{hyp1} is verified), as well as a comparison result between them which is crucial for the subsequent argument on the martingale property.

\begin{theorem}\label{tadd}
	Let $f\in C\left(\mathbb{R}_+;\mathbb{C}_-\right)$ and assume Condition \ref{c1}.
	\begin{enumerate}[label=\emph{(\roman*)}, ref=(\roman*)]
		\item\label{i1} There exist a continuous global solution $\psi\in C\left(\mathbb{R}_+;\mathbb{C}_-\right)$ of \eqref{Vp} and  a real--valued, continuous global solution $\widebar{\psi}\in C\left(\mathbb{R}_+;\mathbb{R}_-\right)$ of \eqref{Vp_real}.
		\item\label{i2} Given $\psi\in C\left(\mathbb{R}_+;\mathbb{C}_-\right)$ and $\widebar{\psi}\in C\left(\mathbb{R}_+;\mathbb{R}_-\right)$  satisfying  \eqref{Vp} and \eqref{Vp_real}, respectively, then the following inequality holds
			\begin{equation}\label{le}
			\mathfrak{R}\psi\left(t\right)\le\widebar{\psi}\left(t\right),\quad t\ge0.
		\end{equation}
	\end{enumerate}
\end{theorem}
\begin{proof}
	The proof of \ref{i1} is in \ref{B1}, and the one of \ref{i2} is in \ref{B2}.
\end{proof}
	In what follows, we take two continuous functions $\psi, \widebar{\psi}$ as in Theorem \ref{tadd} \ref{i1} and
	fix $T>0$. We aim to prove the martingale property of the process $\exp\left\{V^T\right\}$, where $V^T$ is given in \eqref{y1}--\eqref{y2}. For this purpose, we define the process $\widebar{V}^T$ as in \eqref{y1}--\eqref{y2}, substituting $\mathfrak{R}{f}$ [resp., $\widebar{\psi}$] for $f$ [resp., $\psi$]. Theorem \ref{t6} shows that $\widebar{V}^T_t=\widetilde{\widebar{V}}^T_t,\,\mathbb{P}-$a.s., for every $t\in\left[0,T\right]$, where of course we define $\widetilde{\widebar{V}}^T$ as in \eqref{y_tilde} with the same substitution as before. 
	It is known that  $\exp\left\{\widebar{V}^T\right\}$ is a true, real--valued martingale. This is due to \cite[Lemma $6.1$]{edu}, which in turn is an interesting application of the Novikov--type condition in \cite[Theorem \upperRomannumeral{4}.$3$]{LM}. The idea of the present section consists in using the expression \eqref{past1} in order to prove the bound $\left|\exp\left\{V^T\right\}\right|\le C\exp\left\{\widebar{V}^T\right\}$ up to indistinguishability for some $C>0$, so that we can conclude that $\exp\left\{V^T\right\}$ is a martingale, too. 
	
	Direct computations based on the Riccati--Volterra equation \eqref{Vp} yield, for every $h>0$,
	\[
	\Delta_h\psi\left(r\right)=\left(\Delta_h\left(F\left(\cdot,\psi\left(\cdot\right)\right)\right)\ast K\right)\left(r\right)+\left(F\left(\cdot,\psi\left(\cdot\right)\right)\ast \Delta_rK\right)\left(h\right),\quad r\ge0.
	\] 
	Focusing on the second addend on the right side, if we convolve with $L$ we end up with
	\begin{multline*}
		\left(\left(F\left(\diamond,\psi\left(\diamond\right)\right)\ast \Delta_\cdot K\right)\left(h\right)\ast L\right)\left(r\right)=\int_{0}^{r}\left[\int_{0}^{h}F\left(s,\psi\left(s\right)\right)\left(\Delta_{r-u}K\right)\left(h-s\right)\dd s\right] L\left(\dd u\right)\\
		=\int_{0}^{h}F\left(s,\psi\left(s\right)\right)\left(\left(\Delta_{h-s}K\right)\ast L\right)\left(r\right)\dd s,\quad r\ge0,
	\end{multline*}
	where the application of  Fubini's theorem is justified because $K$ is nonnegative and nonincreasing and $L$ is a nonnegative measure. Whence, since $\Delta_h\left(\psi\ast L\right)\left(r\right)=\left(F\left(\cdot,\psi\left(\cdot\right)\right)\ast 1\right)\left({r+h}\right)$, recalling the expression in \eqref{PI_tilde} we can write
	\begin{multline*}
		\widetilde{\Pi}_h\left(r\right)=\left(\Delta_h\psi\ast L\right)\left(r\right)-\Delta_h\left(\psi\ast L\right)\left(r\right)+\int_{0}^{h}F\left(s,\psi\left(s\right)\right)\dd s\\
		=
		-\int_{\left(0,h\right]}\psi\left(h-s\right)L\left(r+\dd s\right)+\int_{0}^{h}F\left(s,\psi\left(s\right)\right)\dd s
		,\quad r\ge0,
	\end{multline*}
	and in particular
	\[
	\mathfrak{R}\left(\widetilde{\Pi}_h\left(r\right)\right)=-\int_{\left(0,h\right]}\mathfrak{R}\left(\psi\left(h-s\right)\right)L\left(r+\dd s\right)+\int_{0}^{h}\mathfrak{R}\left(F\left(s,\psi\left(s\right)\right)\right)\dd s
	,\quad r\ge0.
	\]
	Repeating the same argument for $\widebar{\psi}$ we also obtain
	\[
	\widetilde{\widebar{\Pi}}_h\left(r\right)
	=
	-\int_{\left(0,h\right]}\widebar{\psi}\left(h-s\right)L\left(r+\dd s\right)+\int_{0}^{h}\widebar{F}\left(s,\widebar{\psi}\left(s\right)\right)\dd s
	,\quad r\ge0.
	\] 
	Taking the difference between the two previous equations we infer, for every $r\ge 0$,
	\begin{multline}\label{sale}
		\widetilde{\widebar{\Pi}}_h\left(r\right)-\mathfrak{R}\left(\widetilde{\Pi}_h\left(r\right)\right)=
		-\int_{\left(0,h\right]}\left[\widebar{\psi}-\mathfrak{R}\psi\right]\left(h-s\right)L\left(r+\dd s\right)+\int_{0}^{h}\left[\widebar{F}\left(\cdot,\widebar{\psi}\left(\cdot\right)\right)-\mathfrak{R}\left(F\left(\cdot,\psi\left(\cdot\right)\right)\right)\right]\left(s\right)\dd s\\
		=\left(\widebar{\psi}-\mathfrak{R}\psi\right)\left(h\right)L\left(\left\{r\right\}\right)-\int_{\left[0,h\right]}\left[\widebar{\psi}-\mathfrak{R}\psi\right]\left(h-s\right)L\left(r+\dd s\right)+\int_{0}^{h}\left[\widebar{F}\left(\cdot,\widebar{\psi}\left(\cdot\right)\right)-\mathfrak{R}\left(F\left(\cdot,\psi\left(\cdot\right)\right)\right)\right]\left(s\right)\dd s.
	\end{multline}
	Hence, we see that this function is increasing on the interval $\left(0,\infty\right)$  by \eqref{le} in Theorem \ref{tadd} \ref{i2} and Condition~\ref{c1} as soon as $L$ has no point masses in $\left(0,\infty\right)$. We are now in position to prove the next, important result.
	\begin{theorem}\label{bound_t}
		Assume that the kernel $K\in L^2_{\emph{loc}}\left(\mathbb{R}_+;\mathbb{R}\right)$ satisfies the requirements of Corollary \ref{c9} together with Condition \ref{c1}. Then there exists a constant $C>0$ such that
		\begin{equation}\label{bound}
			\left|\exp\left\{V^T_t\right\}\right|\le C \exp\left\{\widebar{V}^T_t\right\},\quad t\in\left[0,T\right],\,\mathbb{P}-\text{a.s.}
		\end{equation}
		In particular, $\left(\exp\left\{V^T_t\right\}\right)_{t\in\left[0,T\right]}$ is a complex--valued martingale.
	\end{theorem}
	\begin{proof}
		First of all note that $\left|\exp\left\{V^T\right\}\right|=\exp\left\{\mathfrak{R}\left(V^T\right)\right\}$. For the reader's convenience, we write the expression for $\mathfrak{R}\left(V^T\right)$ provided by \eqref{past1}
		\begin{multline*}
			\mathfrak{R}\left(V^T_t\right)=\int_{0}^{t}\mathfrak{R}{f}\left(T-s\right) X_s\,\dd s+\int_{0}^{T-t}\mathfrak{R}\left(F\left(s,\psi\left(s\right)\right)\right) g_0\left(T-s\right)\dd s+	\mathfrak{R}\psi\left(T-t\right) L\left(\left\{0\right\}\right)\left(X-g_0\right)\left(t\right)\\-\left(\dd\left(\mathfrak{R}\left(\widetilde{\Pi}_{T-t}\right)\right) \ast g_0\right)\left(t\right)+\left(\dd\left(\mathfrak{R}\left(\widetilde{\Pi}_{T-t}\right)\right)\ast X\right)\left(t\right),\quad \text{for a.e. }t\in\left(0,T\right),\, \mathbb{P}-\text{a.s.}
		\end{multline*}
		The idea of the proof is simply to compare, term by term, the addends of this sum with the corresponding ones in the expansion of $\widebar{V}^T$ according to \eqref{past1}. We are going to consider a common set $\Omega_0\subset \Omega$, with $\mathbb{P}\left(\Omega_0\right)=1$, such that both the expressions for $\mathfrak{R}\left(V^T_t\right)$ and $\widebar{V}^T_t$ are valid on $\left(0,T\right)\setminus N^\omega$, being $N^\omega\subset \left(0,T\right)$ a $\dd t-$null set for every $\omega\in\Omega_0$.
		
		Regarding the random terms, recall that $X\ge0,\,\mathbb{P}\otimes \dd t-$a.e. Therefore, without loss of generality, we can assume that for every $\omega\in \Omega_0$ and $t\in\left(0,T\right)\setminus N^\omega$ we have $X_t\left(\omega\right)\ge0$. As a consequence (by \eqref{sale})
		\[
		\left(\dd\left(\widetilde{\widebar{\Pi}}_{T-t}-\mathfrak{R}\left(\widetilde{\Pi}_{T-t}\right)\right) \ast X_\cdot\left(\omega\right)\right)\left(t\right)\ge0
		\,\,\,\Longrightarrow \,\,\, \left(\dd\left(\mathfrak{R}\left(\widetilde{\Pi}_{T-t}\right)\right) \ast X_\cdot\left(\omega\right)\right)\left(t\right)\le\left(\dd{\widetilde{\widebar{\Pi}}_{T-t}}\ast X_\cdot\left(\omega\right)\right)\left(t\right).
		\]
		It is important to stress the fact that such an inequality can be stated because the measure $L$ is absolutely continuous with respect to the  Lebesgue measure on the interval $\left(0,\infty\right)$. Summing up, 
		\[
		\left(\dd\left(\mathfrak{R}\left(\widetilde{\Pi}_{T-t}\right)\right) \ast X\right)\left(t\right)\le\left(\dd{\widetilde{\widebar{\Pi}}_{T-t}} \ast X\right)\left(t\right),\quad t\in\left(0,T\right)\setminus N^\omega, \,\omega\in\Omega_0.
		\]
		Moreover, since $L\left(\left\{0\right\}\right)\ge0$, by \eqref{le} we immediately have
		\[
		\mathfrak{R}\psi\left(T-t\right)L\left(\left\{0\right\}\right)X_t\le 	\widebar{\psi}\left(T-t\right)L\left(\left\{0\right\}\right)X_t,\quad t\in\left(0,T\right)\setminus N^\omega, \,\omega\in\Omega_0.
		\]
		The other random addend $\int_{0}^{t}\mathfrak{R}{f}\left(T-s\right)X_s\,\dd s$ appears in both the expressions for $\mathfrak{R}\left(V^T_t\right)$ and $\widebar{V}^T_t$, so it does not need to be discussed.
		
		As for the  deterministic terms, we observe that, by {H\"older's inequality},
		\[
		\left|\int_{0}^{T-t}g_0\left(T-s\right)\left(\mathfrak{R}\left(F\left(s,\psi\left(s\right)\right)\right)-\widebar{F}\left(s,\widebar{\psi}\left(s\right)\right)\right)\dd s\right| \le
		\norm{g_0}_{L^2\left(\left[0,T\right]\right)}\norm{\mathfrak{R}\left(F\left(\cdot,\psi\left(\cdot\right)\right)\right)-\widebar{F}\left(\cdot,\widebar{\psi}\left(\cdot\right)\right)}_{L^2\left(\left[0,T\right]\right)},
		\]
		for any $t\in\left(0,T\right).$
		Hence, calling $C_1= \norm{g_0}_{L^2\left(\left[0,T\right]\right)}\norm{\mathfrak{R}\left(F\left(\cdot,\psi\right)\right)-\widebar{F}\left(\cdot,\psi\right)}_{L^2\left(\left[0,T\right]\right)}$, we have
		\[
		\int_{0}^{T-t}g_0\left(T-s\right)\mathfrak{R}\left(F\left(s,\psi\left(s\right)\right)\right)\dd s
		\le C_1+
		\int_{0}^{T-t}g_0\left(T-s\right)\widebar{F}\left(s,\widebar{\psi}\left(s\right)\right)\dd s,\quad t\in\left(0,T\right).
		\] 
		Furthermore, recalling the continuity of $\psi,\,\widebar{\psi}$ and $g_0$, we call $C_2= \max_{t\in\left[0,T\right]}\left\{\left|\widebar{\psi}-\mathfrak{R}\psi\right|\left(T-t\right)g_0\left(t\right)\right\}$, so that we have
		\[
		-\mathfrak R\psi\left(T-t\right)L\left(\left\{0\right\}\right)g_0\left(t\right)\le L\left(\left\{0\right\}\right)C_2
		-\widebar{\psi}\left(T-t\right)L\left(\left\{0\right\}\right)g_0\left(t\right),\quad t\in\left(0,T\right).
		\]
		Finally, looking at \eqref{sale} we compute
		\begin{multline*}
			\left(\dd\left(\widetilde{\widebar{\Pi}}_{T-t}-\mathfrak{R}\left(\widetilde{\Pi}_{T-t}\right)\right)\ast 1\right)\left(t\right)\le 
			-\int_{\left[0,T-t\right]}\left[\widebar{\psi}-\mathfrak{R}\psi\right]\left(T-t-s\right)\left[L\left(t+\dd s\right)-L\left(\dd s\right)\right]\\
			\le2
			\max_{t\in\left[0,T\right]}\left|\widebar{\psi}\left(t\right)-\mathfrak{R}\psi\left(t\right)\right|L\left(\left[0,T\right]\right)= C_3,\quad t\in\left(0,T\right).
		\end{multline*}
		Hence exploiting the continuity of the input curve we conclude
		\[
		\left|	\left(\dd\left(\widetilde{\widebar{\Pi}}_{T-t}-\mathfrak{R}\left(\widetilde{\Pi}_{T-t}\right)\right)\ast g_0\right)\left(t\right)\right|\le C_3\max_{t\in\left[0,T\right]}\left|g_0\left(t\right)\right|,
		\]
		which in turn implies
		\[
		-\left(\dd\left(\mathfrak{R}\left(\widetilde{\Pi}_{T-t}\right)\right)\ast g_0\right)\left(t\right)\le C_3\max_{t\in\left[0,T\right]}\left|g_0\left(t\right)\right|-\left(	\dd\widetilde{\widebar{\Pi}}_{T-t}\ast g_0\right)\left(t\right),\quad t\in\left(0,T\right).
		\]
		Combining all these results we deduce that 
		\begin{equation}\label{43}
			\mathfrak{R}\left(V_t^T\left(\omega\right)\right)\le C_1+L\left(\left\{0\right\}\right)C_2+C_3\max_{t\in\left[0,T\right]}\left|g_0\left(t\right)\right|+\widebar{V}_t^T\left(\omega\right),\quad t\in\left(0,T\right)\setminus N^\omega, \,\omega\in\Omega_0.
		\end{equation}
		Since $N^\omega$ is a null set, its complementary  $\left(N^\omega\right)^c=\left(0,T\right)\setminus N^\omega$ is dense in $\left[0,T\right]$. 
		Recalling the regularity for the trajectories of the processes $\mathfrak{R}\left(V^T\right)$ and $\widebar{V}^T$,
		we can assume that for every $\omega\in\Omega_0$ both the functions $\mathfrak{R}\left(V^T_\cdot\left(\omega\right)\right)$ and $\widebar{V}^T_\cdot\left(\omega\right)$ are c\`adl\`ag in $\left[0,T\right]$ and left--continuous in $T$. Accordingly, we pass to the limit  --from the right in $\left[0,T\right)$ and from the left in $T$-- to deduce, from \eqref{43}, that
		\[
		\mathfrak{R}\left(V_t^T\left(\omega\right)\right)\le C_1+L\left(\left\{0\right\}\right)C_2+C_3\max_{t\in\left[0,T\right]}\left|g_0\left(t\right)\right|+\widebar{V}_t^T\left(\omega\right),\quad t\in\left[0,T\right],\,\,\omega\in\Omega_0,
		\]
		i.e., \eqref{bound} holds true choosing $C= \exp \left\{C_1+L\left(\left\{0\right\}\right)C_2+C_3\max_{t\in\left[0,T\right]}\left|g_0\left(t\right)\right|\right\}$.

		The second statement of the theorem immediately follows from \cite[Lemma $1.4$]{J}, as $\left(\exp\left\{\widebar{V}^T\right\}\right)_{t\in\left[0,T\right]}$ is a real--valued martingale. Thus, the proof is complete.
	\end{proof}
	Combining Theorem \ref{bound_t} with Theorem \ref{t6} (see \eqref{final}) we conclude the following uniqueness result.
	\begin{corollary}
		The weak solution $X$ of \eqref{f1} is unique in law in $L^2_\emph{loc}\left(\mathbb{R}_+\right)$, that is: if $Y=\left(Y_t\right)_{t\ge0}$ is another predictable process (defined on a possibly different stochastic basis) such that $Y\ge0,\,\mathbb{P}\otimes \dd t-$a.e., which satisfies \eqref{f1}, then the laws of $X$ and $Y$ on the spaces $L^2\left(\left[0,T\right]\right),\,T>0,$ are the same.
	\end{corollary}
	\begin{proof}
		Fix $T>0$ and consider another weak solution $Y$ of \eqref{f1}. We assume that $X$ and $Y$ are defined on the same stochastic basis to keep notation simple. The paths of $Y$ are in $L^2\left(\left[0,T\right]\right),\,\mathbb{P}-$a.s., by Lemma \ref{L2}. We want to show that
		\begin{equation}\label{char}
			\mathbb{E}\left[\exp\left\{i\int_{0}^{T}f\left(s\right)X_s\,\dd s\right\}\right]=
			\mathbb{E}\left[\exp\left\{i\int_{0}^{T}f\left(s\right)Y_s\,\dd s\right\}\right],\quad f\in L^2\left(\left[0,T\right]\right).
		\end{equation}
		First, we verify the previous equation for $f\in C\left(\left[0,T\right]\right)$. Denoting by $\widetilde{f}\left(s\right)= if\left(T-s\right),\,s\in\left[0,T\right],$ by Theorem \ref{t6} and Theorem \ref{bound_t} we have
		\begin{multline*}
			\mathbb{E}\left[\exp\left\{i\int_{0}^{T}f\left(s\right)X_s\,\dd s\right\}\right]=
			\mathbb{E}\left[\exp\left\{\int_{0}^{T}\widetilde{f}\left(T-s\right)X_s\,\dd s\right\}\right]=\mathbb{E}\left[\exp\left\{V_0^T\right\}\right]\\=
			\mathbb{E}\left[\exp\left\{\int_{0}^{T}\widetilde{f}\left(T-s\right)Y_s\,\dd s\right\}\right]=
			\mathbb{E}\left[\exp\left\{i\int_{0}^{T}f\left(s\right)Y_s\,\dd s\right\}\right]
		\end{multline*}
		where we use the fact that $V_0^T$ in \eqref{y2} does not depend on the solution process, but only on the solution of the Riccati--Volterra equation. Therefore \eqref{char} holds true for continuous functions.  Since $C\left(\left[0,T\right]\right)$ is dense in $L^2\left(\left[0,T\right]\right)$, H\"older's inequality allows to carry out a dominated convergence argument that let us recover \eqref{char} for all $f\in L^2\left(\left[0,T\right]\right)$.
		Hence, the laws of $X$ and $Y$ are the same on the space $L^2\left(\left[0,T\right]\right)$ by, for instance, \cite[Proposition $2.5$, Chapter $2$]{DPZ}. This completes the proof.
	\end{proof}
	
	\appendix 
	\section{The forward process}\label{ap_A}
	Given a kernel $K\in L^2_\text{loc}\left(\mathbb{R}_+;\mathbb{R}^{m\times d}\right),$ we want to find an expression for the forward process $$\mathbb{E}\left[X_T\big|\mathcal{F}_t\right],\quad0\le t\le T,$$ for almost every $T\in\mathbb{R}_+$.
	
	If $b\equiv0$, then \eqref{2} implies
	\[
	X_T=g_0\left(T\right)+\left(K\ast \dd \widetilde{Z}\right)_T=g_0\left(T\right)+\int_{0}^{T}K\left(T-s\right)\dd \widetilde{Z}_s,\quad \mathbb{P}-\text{a.s., for a.e. }T\in\mathbb{R}_+.
	\] 
	By the martingale property ensured by \eqref{1.4} we immediately infer that, for almost every $T\in\mathbb{R}_+$,
	\begin{equation}\label{for_no b}
		\mathbb{E}\left[X_T\big|\mathcal{F}_t\right]=g_0\left(T\right)+\int_{0}^{t}K\left(T-s\right)\dd \widetilde{Z}_s,\quad \mathbb{P}-\text{a.s., }t\in\left[0,T\right].
	\end{equation}
	
	If $b\neq 0$, then we consider $m=d$ and introduce the resolvent of the second kind $R_B$ associated with $-KB$. Note that $R_B\in L^2_{\text{loc}}\left(\mathbb{R}_+;\mathbb{R}^{d\times d}\right)$ by \cite[Theorem $3.5$, Chapter $2$]{g}. 
	Convolving \eqref{2} with $R_B$ and \cite[Theorem $2.2$\,(\lowerRomannumeral{8}), Chapter $2$]{g} yield
	\begin{equation*}
		\left(R_B\ast X\right)\left(T\right)=\left(R_B\ast g_0\right)\left(T\right)+\left(\left(R_B\ast K\right)\ast b\left(X\right)\right)\left(T\right)+\left(R_B\ast \left(K\ast \dd \widetilde{Z}\right)\right)\left(T\right),\quad \text{for a.e. }T\in\mathbb{R}_+,\,\mathbb{P}-\text{a.s.}
	\end{equation*}
	The associativity of the stochastic convolution  proved in Lemma \ref{4.1} (with $\rho=R_B$) and the joint measurability of the processes involved let us rewrite this equality as follows
	\begin{multline}{\label{aft}}
		\left(R_B\ast X\right)\left(T\right)=\left(R_B\ast g_0\right)\left(T\right)+\left(\left(R_B\ast K\right)\ast b_0\right)\left(T\right)
		+\left(\left(R_B\ast KB\right)\ast X\right)\left(T\right)\\+\left(\left(R_B\ast K\right)\ast \dd \widetilde{Z}\right)_T,\quad \mathbb{P}-\text{a.s.},\text{ for a.e. }T\in\mathbb{R}_+.
	\end{multline}
	From the resolvent identity (see the footnote \addtocounter{footnote}{-1}\addtocounter{Hfootnote}{-1}\footnotemark) we have  $R_B\ast KB=KB+R_B$ a.e.  in $\mathbb{R}_+$, so we  rewrite  Equation \eqref{aft} as follows
	\begin{equation}\label{4}
		0=\left(R_B\ast g_0\right)\left(T\right)+\left(\left(R_B\ast K\right)\ast b_0\right)\left(T\right)
		+\left(KB\ast X\right)\left(T\right)+\left(\left(R_B\ast K \right)\ast \dd \widetilde{Z}\right)_T,\quad \mathbb{P}-\text{a.s.},\text{ for a.e. }T\in\mathbb{R}_+.
	\end{equation}
	Let us consider the \emph{canonical resolvent} $E_B= K-R_B\ast K$; subtracting \eqref{4} from \eqref{2} we have
	\begin{equation*}
		X_T=\left(g_0-\left(R_B\ast g_0\right)\right)\left(T\right)+\left(E_B\ast b_0\right)\left(T\right)+\left(E_B\ast \dd \widetilde{Z}\right)_T,\quad \mathbb{P}-\text{a.s.},\text{ for a.e. }T\in\mathbb{R}_+.
	\end{equation*}
	Hence by the martingale property guaranteed by \eqref{1.4} we are able to find an expression for the forward process $\mathbb{E}\left[X_T\big|\mathcal{F}_t\right]$, namely for almost every $T\in\mathbb{R}_+,$ for every $t\in\left[0,T\right]$ it holds
	\begin{equation} \label{forward}
		\mathbb{E}\left[X_T\big|\mathcal{F}_t\right]=
		\left(g_0\left(T\right)-\left(R_B\ast g_0\right)\left(T\right)\right)+
		\left(E_B\ast b_0\right)\left(T\right)
		+\int_{0}^{t}E_B\left(T-s\right)\dd \widetilde{Z}_s, \quad \mathbb{P}-\text{a.s.}
	\end{equation}
	
	Finally, notice that  \eqref{forward} reduces to \eqref{for_no b} as $b\equiv0.$ Indeed, since $E_0=K$ a.e. in $\mathbb{R}_+$ as $R_0=0\left(\in\mathbb{R}^{d\times d}\right)$,  combining  \eqref{2} with Lemma \ref{l1} (see \eqref{2.1}) we have    
	\begin{equation}\label{serve_lei}
		X_T=g_0\left(T\right)+\left(E_0\ast \dd \widetilde{Z}\right)_T,\quad \mathbb{P}-\text{a.s., for a.e. }T\in\mathbb{R}_+,
	\end{equation}
	and the assertion follows by the martingale property.
	\begin{rem}
		Equation \eqref{forward} with $t=0$ implies that $\mathbb{E}\left[X_T\right]=\left(g_0-\left(R_B\ast g_0\right)\right)\left(T\right)+\left(E_B\ast b_0\right)\left(T\right)$ for a.e. $T\in\mathbb{R}_+$. This result can be confirmed with a direct method. Specifically, by \eqref{1.4} and Tonelli's theorem the function $\mathbb{E}\left[\left|X_\cdot\right|\right]\in L^1_\emph{loc}\left(\mathbb{R}_+;\mathbb{R}\right)$. Hence taking expectations in \eqref{f} we obtain, by Fubini's theorem, 
		\begin{equation*}
			\mathbb{E}\left[X_T\right]=
			\left(g_0+K\ast b_0\right)\left(T\right)+\left(KB\ast \mathbb{E}\left[X_\cdot\right]\right)\left(T\right),\quad  \text{for  a.e.  }T\in\mathbb{R}_+,
		\end{equation*}
		i.e., $\mathbb{E}\left[X_\cdot\right]+\left(\left(-KB\right)\ast\mathbb{E}\left[X_\cdot\right]\right)=g_0+K\ast b_0$ a.e. in $\mathbb{R}_+$.
		An application of the variation of constants formula in \cite[Theorem $3.5$, Chapter $2$]{g} let us conclude
		\[
		\mathbb{E}\left[X_T\right]=\left(g_0-\left(R_B\ast g_0\right)+\left(E_B\ast b_0\right)\right)\left(T\right),\quad \text{for a.e. }T\in\mathbb{R}_+,
		\] 
		as desired.
	\end{rem}
	\section{On the {$1-$}dimensional, deterministic Riccati--Volterra equation}\label{ap_B}
	Here we focus on the Riccati--Volterra equation used in Section \ref{sec:example}, i.e., \eqref{Vp} with
	\begin{equation}\label{F_ex}
		F\left(t,u\right)= f\left(t\right)+bu+\frac{c}{2}u^2+\int_{\mathbb{R}_+}\left(e^{u\xi}-1-u\xi\right)\nu\left(\dd\xi\right),\quad \left(t,u\right)\in \mathbb{R}_+\times\mathbb{C}_-,
	\end{equation}
	where $f\in C\left(\mathbb{R}_+;\mathbb{C}_-\right).$
	Throughout the section, we require Condition \ref{c1} on the kernel $K$. 
	\subsection{Existence of a global solution}\label{B1}
	It is easy to argue that \eqref{Vp} admits a continuous, noncontinuable solution $\psi$, with $\mathfrak{R}\psi\le0$, defined on the maximal interval 
	$\left[0,T_{\text{max}}\right)$ (see \cite[Theorem $2.5$, Step $1$]{edu}). We are concerned with showing that $T_\text{max}=\infty$, i.e., that $\psi$ 	does not explode in finite time (cfr. \cite[Theorem $1.1$, Chapter $12$]{g}). Let us fix a generic $T\in\left(0,T_\text{max}\right)$; taking real and imaginary parts in \eqref{Vp} and \eqref{F_ex} we have, on the interval $\left[0,T\right]$,
	\begin{align}
		\label{real}
		&\mathfrak{R}\psi=
		K\ast \left[\mathfrak{R}f+b\,\mathfrak{R}\psi+\frac{c}{2}\left(\left|\mathfrak{R}\psi\right|^2-\left|\mathfrak{Im}\psi\right|^2\right)+
		\int_{\mathbb{R}_+}
		\left(\cos\left(\mathfrak{Im}\psi\cdot \xi\right)e^{\mathfrak{R}\psi\cdot \xi}-1-\mathfrak{R}\psi\cdot \xi\right)\nu\left(\dd \xi\right)
		\right],\\
		\label{im}&
		\mathfrak{Im}\psi=
		K\ast
		\left[\mathfrak{Im}f+b\,\mathfrak{Im}\psi+c\,\mathfrak{R}\psi\,\mathfrak{Im}\psi
		+
		\int_{\mathbb{R}_+}\left(\sin\left(\mathfrak{Im}\psi\cdot \xi\right)e^{\mathfrak{R}\psi\cdot \xi}-\mathfrak{Im}\psi\cdot  \xi\right)\nu\left(\dd \xi\right)
		\right].
	\end{align}
	First we study the imaginary part. In particular, we consider the function $h\colon \mathbb{R}_-\times \mathbb{R}\to \mathbb{R}$ defined as follows
	\begin{equation*}
		h\left(x,y\right)=\begin{cases}
			\frac{1}{y}\int_{\mathbb{R}_+}\left(\sin\left(y\, \xi\right)-y\, \xi\right)e^{x\cdot \xi}\nu\left(\dd \xi\right),&y\neq0\\
			0,&y=0
		\end{cases},\quad x\le 0.
	\end{equation*}
	Note that $h$ is continuous and non--positive in its domain. By construction 
	\[
	y\cdot h\left(x,y\right)=\int_{\mathbb{R}_+}\left(\sin\left(y\, \xi\right)-y\, \xi\right)e^{x\cdot \xi}\nu\left(\dd \xi\right),\quad \left(x,y\right)\in\mathbb{R}_-\times \mathbb{R}.
	\]
	Hence we can use this function to rewrite \eqref{im} in the following form
	\begin{align*}
		\mathfrak{Im}\psi&=
		K\ast
		\left[\mathfrak{Im}f+b\,\mathfrak{Im}\psi+c\,\mathfrak{R}\psi\,\mathfrak{Im}\psi
		+
		\left(	\int_{\mathbb{R}_+}\xi\left(e^{\mathfrak{R}\psi\cdot \xi}-1\right)\nu\left(\dd \xi\right)\right)\mathfrak{Im}\psi
		+
		h\left(\mathfrak{R}\psi,\mathfrak{Im}\psi\right)\mathfrak{Im}\psi
		\right],
	\end{align*}
	which holds true on $\left[0,T\right]$. Let us consider the unique, continuous, non--negative solution on $\left[0,T\right]$ of the linear equation 
	\[
	g=K\ast\left[
	\left|\mathfrak{Im}\,f\right|+b\,g+\left(c\,\mathfrak{R}\psi+
	\int_{\mathbb{R}_+}\xi\left(e^{\mathfrak{R}\psi\cdot \xi}-1\right)\nu\left(\dd\xi\right)+
	h\left(\mathfrak{R}\psi,\mathfrak{Im}\psi\right)
	\right)g
	\right].
	\]
	By Condition \ref{c1}, we can invoke \cite[Theorem C.$1$]{ee} to deduce that $\left|\mathfrak{Im}\psi\right|\le g$ on $\left[0,T\right]$. Next we introduce $u$, the unique, continuous  solution of the linear equation 
	\[
	u=K\ast\left[\left|\mathfrak{Im}f\right|+b\,u\right].
	\]
	Notice that $u$ is defined on $\mathbb{R}_+$, and that $g\le u$ on $\left[0,T\right]$ (again by \cite[Theorem C.$1$]{ee}), as in this interval one has
	\[
	c\,\mathfrak{R}\psi 
	+
	\int_{\mathbb{R}_+}\xi 
	\left(e^{\mathfrak{R}\psi\cdot\xi}-1\right)\nu\left(\dd\xi\right)+
	h\left(\mathfrak{R}\psi,\mathfrak{Im}\psi\right)\le 0.
	\]
	Therefore we have obtained the bound
	\begin{equation}\label{bound_im}
		\left|\mathfrak{Im}\psi\left(t\right)\right|\le u\left(t\right),\quad 0\le t\le T.
	\end{equation}
	
	Secondly,  Equation \eqref{real} ensures that $\mathfrak{R}\psi$ satisfies 
	\begin{multline*}
		\mathfrak{R}\psi
		=
		K\ast 
		\Bigg[
		\mathfrak{R}\,f+b\,\mathfrak{R}\psi+\frac{c}{2}\left(\left|\mathfrak{R}\psi\right|^2-\left|\mathfrak{Im}\psi\right|^2\right)
		+
		\int_{\mathbb{R}_+}\left(e^{\mathfrak{R}\psi\cdot \xi}-1-\mathfrak{R}\psi\cdot \xi\right)\nu\left(\dd \xi\right)
		\\-
		\left|	\int_{\mathbb{R}_+}e^{\mathfrak{R}\psi\cdot \xi}\left(\cos\left(\mathfrak{Im}\psi\cdot \xi\right)-1\right)\nu\left(\dd \xi\right)\right|
		\Bigg]
	\end{multline*} 
	on $\left[0,T\right]$. Since $\left|\cos\left(x\right)-1\right|=1-\cos \left(x\right)\le{x^2}/{2},\,x\in\mathbb{R}$, we have (also recalling \eqref{bound_im})
	\begin{equation}\label{b_1}
		\left|	\int_{\mathbb{R}_+}e^{\mathfrak{R}\psi\cdot \xi}\left(\cos\left(\mathfrak{Im}\psi\cdot \xi\right)-1\right)\nu\left(\dd \xi\right)\right|\le \frac{1}{2}\left(\int_{\mathbb{R}_+} \left|\xi\right|^2\nu\left(\dd\xi\right)\right)\left|\mathfrak{Im}\psi\right|^2\le \frac{1}{2}\left(\int_{\mathbb{R}_+} \left|\xi\right|^2\nu\left(\dd \xi\right)\right)u^2,\quad \text{on }\left[0,T\right].
	\end{equation}
	This suggests to introduce the linear equation
	\[
	l=K\ast\left[\mathfrak{R}\,f+b\,l-\left(\frac{c}{2}+\frac{1}{2}\int_{\mathbb{R}_+}\left|\xi\right|^2\nu\left(\dd\xi\right)\right)u^2\right],
	\]
	which has a unique, continuous, non--positive solution $l$ defined on the whole $\mathbb{R}_+.$ At this point, observe that the difference $\mathfrak{R}\psi-l$ satisfies the linear equation
	\begin{multline*}
		\chi=K\ast
		\Bigg[b\,\chi+\frac{c}{2}\left|\mathfrak{R}\psi\right|^2+\frac{c}{2}\left(u^2-\left|\mathfrak{Im}\psi\right|^2
		\right)+
		\int_{\mathbb{R}_+}\left(e^{\mathfrak{R}\psi\cdot \xi}-1-\mathfrak{R}\psi\cdot \xi\right)\nu\left(\dd \xi\right)\\
		+\left(\frac{1}{2}\left(\int_{\mathbb{R}_+} \left|\xi\right|^2\nu\left(\dd \xi\right)\right)\,u^2-	\left|\int_{\mathbb{R}_+}e^{\mathfrak{R}\psi\cdot \xi}\left(\cos\left(\mathfrak{Im}\psi\cdot \xi\right)-1\right)\nu\left(\dd \xi\right)\right|\right)
		\Bigg].
	\end{multline*}
	It admits a unique, continuous solution on $\left[0,T\right]$ which is non--negative by \eqref{bound_im}, \eqref{b_1} and the fact that $e^x-1-x\ge0,\,x\in\mathbb{R}.$ Since $T\in\left(0,T_{\text{max}}\right)$ was chosen arbitrarily, we conclude that
	\[
	l\left(t\right)\le \mathfrak{R}\psi\left(t\right)\le0\text{\quad and \quad  }\left|\mathfrak{Im}\psi\left(t\right)\right|\le u\left(t\right),\quad 0\le t< T_\text{max}.
	\]
	Now recalling that $l$ and $u$ are continuous in $\mathbb{R}_+$, so they are bounded in every compact interval, we conclude that $T_\text{max}=\infty$, as desired.
	
	Finally we notice that if $f$ takes values in $\mathbb{R}_-$, then from \eqref{bound_im} we deduce that any solution of \eqref{Vp} is real--valued, as well. In particular, $\widebar{\psi}$ in \eqref{Vp_real} is $\mathbb{R}_--$valued.
	\subsection{A comparison result}\label{B2}
	The goal of this appendix is to prove the inequality \eqref{le} in Theorem \ref{tadd} \ref{i2}, which is of utmost importance for the argument in Section \ref{sec:example}. Precisely, we want to show that
	\begin{equation*}
		\mathfrak{R}\psi\left(t\right)\le\widebar{\psi}\left(t\right),\quad t\ge0,
	\end{equation*}
where $\psi\in C\left(\mathbb{R}_+;\mathbb{C}_-\right)$ and $\widebar{\psi}\in C\left(\mathbb{R}_+;\mathbb{R}_-\right)$  satisfy \eqref{Vp} and \eqref{Vp_real}, respectively. 
	Direct computations based on the definitions in \eqref{F_ex} and \eqref{Fbar} show that, for every $u\in\mathbb{C}_-$ and $t\ge0$,
	\begin{align*}
		\mathfrak{R}\left(F\left(t,u\right)\right)&
		=\mathfrak{R}f\left(t\right)+b\mathfrak{R}\left(u\right)
		+\frac{c}{2}\left(\left|\mathfrak{R}\left(u\right)\right|^2-\left|\mathfrak{Im}\left(u\right)\right|^2\right)
		+\int_{\mathbb{R}_+}\left(\cos\left(\mathfrak{Im}\left(u\right)\xi\right)e^{\mathfrak{R}\left(u\right)\xi}-1-\mathfrak{R}\left(u\right)\xi\right)\nu\left(\dd\xi\right)\\&
		\le \mathfrak{R}{f}\left(t\right)+b\mathfrak{R}\left(u\right)+\frac{c}{2}\left|\mathfrak{R}\left(u\right)\right|^2
		+\int_{\mathbb{R}_+}\left(e^{\mathfrak{R}\left(u\right)\xi}-1-\mathfrak{R}\left(u\right)\xi\right)\nu\left(\dd\xi\right)
		=\widebar{F}\left(t,\mathfrak{R}\left(u\right)\right).
	\end{align*}
	Summarizing, 
	\begin{equation}\label{est}
		\mathfrak{R}\left(F\left(t,u\right)\right)\le\widebar{F}\left(t,\mathfrak{R}\left(u\right)\right),\quad u\in\mathbb{C}_-,\,t\ge0.
	\end{equation}
	Then taking the real parts in \eqref{Vp} and recalling that --under Condition \ref{c1}-- the kernel $K$ is nonnegative on $\left(0,\infty\right)$, we obtain
	\begin{equation*}
		\mathfrak{R}\left(\psi\left(t\right)\right)\le\int_0^t K\left(t-s\right)\widebar{F}\left(s,\mathfrak{R}\left(\psi\left(s\right)\right)\right)\dd s,\quad t\ge0.
	\end{equation*}
	Therefore we can introduce a nonnegative function $\gamma\colon\mathbb{R}_+\to\mathbb{R}_+$ defined by the  relation
	\begin{equation}\label{j1}
		\mathfrak{R}\left(\psi\left(t\right)\right)=-\gamma\left(t\right)+\int_0^t K\left(t-s\right)\widebar{F}\left(s,\mathfrak{R}\left(\psi\left(s\right)\right)\right)\dd s,\quad t\ge0;
	\end{equation}
	we immediately note that, using \eqref{Vp}, one can rewrite $\gamma$ as follows
	\begin{equation}\label{j}
		\gamma\left(t\right)=\int_{0}^{t}K\left(t-s\right)\left(\widebar{F}\left(s,\mathfrak{R}\psi\left(s\right)\right)-\mathfrak{R}\left(F\left(s,\psi\left(s\right)\right)\right)\right)	\dd s,\quad t\ge0.
	\end{equation}
	For a generic map $g\colon\mathbb{R}_+\to \mathbb{R}$ consider the condition 
	\begin{equation}\label{B.3}
		\Delta_hg-\left(\Delta_hK\ast L\right)\left(0\right)g-\dd\left(\Delta_hK\ast L\right)\ast g\ge0,\quad h\ge0;
	\end{equation}
	we denote by $\mathcal{G}_K= \left\{g\colon\mathbb{R}_+\to \mathbb{R}\text{ s.t. }g\text{ is continuous, satisfies }\eqref{B.3}\text{ and }g\left(0\right)\ge0\right\}$ the set of admissible curves, see \cite[Condition B.$3$]{ee}, and also \cite[Equations $\left(2.14\right)$-$\left(2.15\right)$]{edu}. By \cite[Remark B.$6$]{ee} and \eqref{j} we infer that $\gamma\in\mathcal{G}_K$.
	
	At this point we subtract \eqref{j1} from \eqref{Vp_real} to deduce, calling $\delta=\widebar{\psi}-\mathfrak{R}\psi$, that
	\begin{equation}\label{delta}
		\delta\left(t\right)=\gamma\left(t\right)+\int_{0}^t K\left(t-s\right)\left(\widebar{F}\left(s,\widebar{\psi}\left(s\right)\right)
		-	\widebar{F}\left(s,\mathfrak{R}\left(\psi\left(s\right)\right)\right)
		\right)\dd s,\quad t\ge0.
	\end{equation}
	We then need to study the increments of $\widebar{F}$ in the second variable. Namely, fix $u_1,u_2\in\mathbb{R}_-$ and use the definition \eqref{Fbar} to write
	\begin{multline*}
		\widebar{F}\left(t,u_1\right)-\widebar{F}\left(t,u_2\right)=b\left(u_1-u_2\right)+\frac{c}{2}\left(u_1^2-u_2^2\right)
		+\int_{\mathbb{R}_+}\left(e^{u_1\xi}
		-	e^{u_2\xi}
		-\left(u_1-u_2\right)\xi\right)\nu\left(\dd\xi\right)	\\
		=\left[b+\frac{c}{2}\left(u_1+u_2\right)\right]\left(u_1-u_2\right)+\int_{\mathbb{R}_+}\left(e^{u_1\xi}
		-	e^{u_2\xi}
		-\left(u_1-u_2\right)\xi\right)\nu\left(\dd\xi\right),\quad t\ge0.
	\end{multline*}
	Hence substituting $\widebar{\psi}$ and  $\mathfrak{R}\psi$ to $u_1$ and $u_2$, respectively, we have
	\begin{equation*}
		\widebar{F}\left(t,\widebar{\psi}\left(t\right)\right)-\widebar{F}\left(t,\mathfrak{R}\left(\psi\left(t\right)\right)\right)
		=\underbrace{\left[b+\frac{c}{2}\left(\widebar{\psi}\left(t\right)+\mathfrak{R}\left(\psi\left(t\right)\right)\right)\right]}_{= {z}\left(t\right)}\delta\left(t\right)+\underbrace{\int_{\mathbb{R}_+}\left(e^{\widebar{\psi}\left(t\right)\xi}
			-e^{\mathfrak{R}\psi\left(t\right)\xi}
			-\delta\left(t\right)\xi\right)\nu\left(\dd\xi\right)}_{= w\left(t\right)}
	\end{equation*}
	for $t\ge 0$. Going back to \eqref{delta}, 
	\begin{equation}\label{delta1}
		\delta\left(t\right)=\gamma\left(t\right)+\int_{0}^tK\left(t-s\right)\left({z}\left(s\right)\delta\left(s\right)+w\left(s\right)\right)\dd s,\quad t\ge0.
	\end{equation}
	We aim to apply \cite[Theorem C.$1$]{ee} in order to conclude $\delta\ge0$ in $\mathbb{R}_+$.
	\begin{itemize}
		\item In the continuous case the integral in $\nu\left(\dd\xi\right)$, i.e., the function $w$, simply disappears, hence the application of \cite[Theorem C.$1$]{ee} is straightforward. 
		\item In the jump case we need to deal with such an integral. Observe that the function $w$ has opposite sign with respect to $\delta$, so there is no hope of applying  \cite[Theorem C.$1$]{ee} unless we modify its expression. Fortunately this can be done using the {mean value theorem}, in combination with simple real--analysis arguments. 
		
		First, for every $\xi >0$ we define $f_{\xi}\left(u\right)= e^{\xi u},\,u\in\mathbb{R},$ so $f_{\xi}'\left(u\right)=\xi e^{\xi u}$. Observe that the derivative $f'_{\xi}$ is continuous and strictly increasing in $\mathbb{R}$, hence its inverse $h_{\xi}= \left(f_\xi'\right)^{-1}$ is continuous on $\left(0,\infty\right)$, as well. By the mean value theorem, for every $u_1,u_2\in\mathbb{R}$ there exists $c_{\xi}\in\left[u_1\wedge u_2, u_1\vee u_2\right]$ such that
		\[
		f_{\xi}\left(u_2\right)-f_{\xi}\left(u_1\right)=f_\xi'\left(c_\xi\right)\left(u_2-u_1\right).
		\] 
		In particular $c_{\xi}\in \left(u_1\wedge u_2, u_1\vee u_2\right)$ when $u_1\neq u_2.$\\
		Secondly, we consider the functions $\widebar{\psi}$ and  $\mathfrak{R}\psi$, and we can say that for every $t\in\mathbb{R}_+$ there exists $c_{\xi}\left( t\right)\in \left[\widebar{\psi}\left(t\right)\wedge\mathfrak{R}\psi\left(t\right),\widebar{\psi}\left(t\right)\vee\mathfrak{R}\psi\left(t\right)\right] $ (in the interior of such interval whenever $\widebar{\psi}\neq\mathfrak{R}\psi$, i.e., whenever $\delta\neq0$) such that 
		\begin{equation}\label{8}
			e^{\xi\widebar{\psi}\left(t\right)}-e^{\xi\mathfrak{R}\psi\left(t\right)}=f'_{\xi}\left(c_\xi\left(t\right)\right)\left(\widebar{\psi}\left(t\right)-\mathfrak{R}\psi\left(t\right)\right)=\xi e^{\xi c_{\xi}\left(t\right)}\delta\left(t\right).
		\end{equation}
		By the axiom of choice we construct the function $c\colon\mathbb{R}_+\times\mathbb{R}_+\to\mathbb{R}_-$ defined by 
		\[
		c\left(\xi,t\right)= c_{\xi}\left(t\right),\quad \xi>0,\,t\ge0,
		\]
		and $c\left(0,t\right)= 0,\,t\ge0.$ Note that the codomain of $c\left(\cdot,\cdot\right)$ is $\mathbb{R}_-$ since both $\widebar{\psi}$ and $\mathfrak{R}\psi$ take values there. Recalling the definition of $w$ and using \eqref{8} we can write
		\[
		w\left(t\right)=
		\int_{\mathbb{R}_+}\left[e^{\widebar{\psi}\left(t\right)\xi}
		-e^{\mathfrak{R}\left(\psi\left(t\right)\right)\xi}
		-\delta\left(t\right)\xi\right]\nu\left(\dd\xi\right)=
		{\color{black}{\underbrace{\left(\int_{\mathbb{R}_+}\xi\left[e^{\xi c\left(\xi,t\right)}-1\right]\nu\left(\dd\xi\right)\right)}_{= \widebar{w}\left(t\right)}}}\,\,\delta\left(t\right),\quad t\ge0.
		\]
		Now we have to prove that $\widebar{w}$ is continuous on $\mathbb{R}_+$. The first step is to show the continuity of the function $c\left(\xi,\cdot\right)$ in $\mathbb{R}_+$ for every fixed $\xi\in \mathbb{R}_+.$ It is of course trivial for $\xi=0$, so we just focus on $\xi>0$. If $\widebar{t}\in\left(0,\infty\right)$ such that $\delta\left(\widebar{t}\right)>0$, then we can find $\epsilon>0$ such that $\delta>0$ in $\left(\widebar{t}-\epsilon,\widebar{t}+\epsilon\right)$. Hence we use \eqref{8} to prove that
		\[
		c\left(\xi,t\right)=h_{\xi}\left(\frac{e^{\xi\widebar{\psi}\left(t\right)}-e^{\xi\mathfrak{R}\psi\left(t\right)}}{\delta\left(t\right)}\right),\quad t\in\left(\widebar{t}-\epsilon,\widebar{t}+\epsilon\right),
		\]
		recalling that $h_{\xi}=\left(f_{\xi}'\right)^{-1}$. So $c\left(\xi,\cdot\right)$ is continuous in the points $\widebar{t}\in\left(0,\infty\right)$ where $\delta\left(\widebar{t}\right)>0$. An analogous reasoning shows the continuity in the points where $\delta<0.$ Let us now consider  $\widebar{t}\in\mathbb{R}_+$ a zero for $\delta$, i.e., $\delta\left(\widebar{t}\right)=0$. For every sequence $\left(t_n\right)_n\subset\mathbb{R}_+$ such that $t_n\to \widebar{t}$ as $n\to \infty$ one has, by construction,
		\[
		\widebar{\psi}\left(t_n\right)\wedge\mathfrak{R}\psi\left(t_n\right)\le c\left(\xi,t_n\right)\le \widebar{\psi}\left(t_n\right)\vee\mathfrak{R}\psi\left(t_n\right),\quad n \in\mathbb{N}.
		\]
		Therefore an application of the {squeeze theorem} gives
		\[
		\lim_{n\to\infty}c\left(\xi,t_n\right)=
		\mathfrak{R}\psi\left(\widebar{t}\right)=\widebar{\psi}\left(\widebar{t}\right)=c\left(\xi,\widebar{t}\right).
		\]
		
		At this point we deduce the continuity of the function $\widebar{w}$ using the dominated convergence theorem. Indeed, take $t\in\mathbb{R}_+$, a sequence $t_n\to t$, and define $g_{\left(n\right)}\left(\xi\right)=\xi\left[e^{\xi c\left(\xi,t_{\left(n\right)}\right)}-1\right],\,\xi\in\mathbb{R}_+$. Then $g_n\to g$ pointwise in $\mathbb{R}_+$ by the continuity of $c\left(\xi,\cdot\right)$ and, for a certain $C>0$ s.t. $t_n\le C, \,n\in\mathbb{N}$ (which exists since $\left(t_n\right)_n$ is bounded), we have 
		\begin{multline*}
			\left|\xi\left[e^{\xi c\left(\xi,t_n\right)}-1\right]\right|
			=\xi\left[1-e^{\xi c\left(\xi,t_n\right)}\right]\le\xi\left[1-e^{\xi \min_{0\le s\le C}c\left(\xi,s\right)}\right]\\
			\le\xi\left[1-e^{\xi \min_{0\le s\le C}\left\{\left(\mathfrak{R}\psi\wedge\widebar{\psi}\right)\left(s\right)\right\}}\right]\in L^1\left(\dd\nu\right),\quad n\in\mathbb{N}.
		\end{multline*}
		Therefore we can rewrite \eqref{delta1} as follows
		\[
		\delta\left(t\right)=\gamma\left(t\right)+\int_{0}^tK\left(t-s\right)\left({z}\left(s\right)+\widebar{w}\left(s\right)\right)\delta\left(s\right)\dd s,\quad t\ge0,
		\]
		and we invoke \cite[Theorem C.$1$]{ee} to assert that $\delta\ge0,$ i.e., that \eqref{le} holds true.
	\end{itemize}

\end{document}